\documentclass[11pt]{article}

\usepackage{geometry}
\usepackage{graphicx}
\usepackage{epstopdf}
\usepackage{booktabs} % for much better looking tables
\usepackage{array} % for better arrays (eg matrices) in maths
\usepackage{paralist} % very flexible & customisable lists (eg. enumerate/itemize, etc.)
\usepackage{verbatim} % adds environment for commenting out blocks of text & for better verbatim
\usepackage{subfigure} % deprecated, use caption and subcaption
\usepackage{caption}
\usepackage{fancyhdr} % This should be set AFTER setting up the page geometry
\usepackage[nottoc,notlof,notlot]{tocbibind}
\usepackage[titles,subfigure]{tocloft}

\usepackage{amsmath,amsfonts,amsthm,mathrsfs,amssymb,cite}
\usepackage[usenames]{color}
\usepackage{mathtools}
\mathtoolsset{showonlyrefs}

\newtheorem{thm}{Theorem}[section]

\theoremstyle{definition}

\newtheorem{rem}{Remark}[section]
\newtheorem{example}{Example}
\numberwithin{equation}{section}

\title{\bf Mathematical design of a novel input/instruction device using a moving emitter}

\author{Yukun  Guo\thanks{Department of Mathematics, Harbin Institute of Technology, Harbin, P. R. China. Email: {\tt ykguo@hit.edu.cn}},
\and Jingzhi Li\thanks{Department of Mathematics, Southern University of Science and Technology, Shenzhen, P. R. China. Email: {\tt li.jz@sustc.edu.cn}},
\and Hongyu Liu\thanks{Department of Mathematics, Hong Kong Baptist University, Kowloon, Hong Kong SAR, P. R. China.  Email:  {\tt hongyuliu@hkbu.edu.hk}}
\and Xianchao Wang\thanks{Department of Mathematics, Harbin Institute of Technology, Harbin, P. R. China. Email: {\tt xcwang90@gmail.com}}}

\date{} % Activate to display a given date or no date (if empty),
\begin{document}
\maketitle

\begin{abstract}

This paper is concerned with the mathematical design of a novel input/instruction device using a moving emitter. The emitter generates a point source and can be installed on a digit pen or worn on the finger of the human being who wants to interact/communicate with the computer. The input/instruction can be recognized by identifying the motion trajectory of the emitter performed by the human being from the collected wave field data. The identification process is modelled as an inverse source problem where one intends to identify the trajectory of a moving point source. There are several salient features of our study which distinguish our result from the existing ones in the literature. First, the point source is moving in an inhomogeneous background medium, which models the human body. Second, the dynamical wave field data are collected in a limited aperture. Third, the reconstruction method is independent of the background medium, and it is totally direct without any matrix inversion. Hence, it is efficient and robust with respect to the measurement noise. Both theoretical justifications and computational experiments are presented to verify our novel findings.

\medskip

\medskip



\end{abstract}

\section{Introduction}\label{sect:1}

\subsection{Background and motivation}

We are concerned with input/instruction technologies to interact and communication with the computer. Generally, there are three major ingredients for a input/instruction technology: the computing machine, the input/instruction device and the human being who performs the specific inputs/instructions to the computer. The input/instruction device bridges the computer and the human being. Nowadays, the majority of input/instruction technology is based on {\it text user interfaces} (TUIs) or {\it graphical user interfaces} (GUIs), using devices such as keyboards, mouse or touch-screens. There are some emerging input/instruction technologies including the gesture recognition. The gesture recognition enable humans to communicate with the machine and interact naturally without any mechanical devices. Many approaches have been proposed, e.g., using cameras to capture the human body gesture and using computer vision algorithms to interpret the body language; see \cite{Wiki} for  relevant discussions. In \cite{LWY}, a novel gesture recognition approach was proposed and investigated by adopting  techniques from the inverse wave scattering theory. In this paper, we propose and investigate an alternative and novel input/instruction technology using  the inverse source reconstruction techniques; see Figure~\ref{fig.1} for a schematic illustration.
\begin{figure}[t]
\centering
  \setlength{\unitlength}{0.05 mm}%
  \begin{picture}(1380.7, 819.5)(0,0)
  \put(0,0){\includegraphics{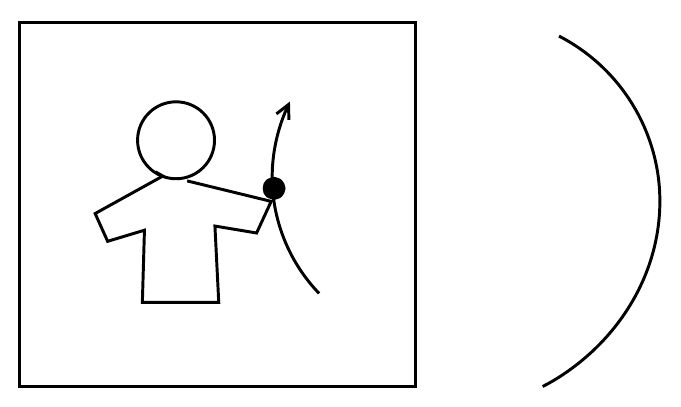}}
  \put(332.72,297.04){\fontsize{14.23}{17.07}\selectfont \makebox(50.0, 100.0)[l]{$\Omega$\strut}}
  \put(602.61,438.21){\fontsize{14.23}{17.07}\selectfont \makebox(50.0, 100.0)[l]{$z_0(t)$\strut}}
  \put(720.94,622.98){\fontsize{14.23}{17.07}\selectfont \makebox(50.0, 100.0)[l]{$D$\strut}}
  \put(1200.51,402.92){\fontsize{14.23}{17.07}\selectfont \makebox(50.0, 100.0)[l]{$\Gamma$\strut}}
  \end{picture}%
  \caption{Schematic illustration of the proposed input/instruction technology using a moving emitter.}\label{fig.1}
\end{figure}
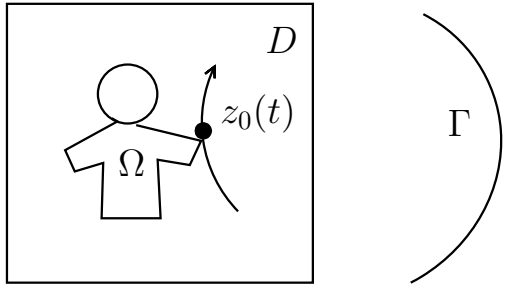

In the proposed technology, in order to give the specific input or instruction to the computer, the person uses an active emitter and moves it following a certain path. The emitter generates a moving point source and there is a sensor monitoring the wave field generated by the emitter away from it on a surface $\Gamma$ in a timely manner. The sensor recognizes the input/instruction via identifying the motion trajectory of the point wave. Using the concept of such a trajectory identifying technology, it is possible to input texts into the computer by writing in the air as what has been used for a touch-screen by tapping on the screen. For the practical design, the emitter could be a ring wearing on the finger or a digital pen, and the sensor could be installed on the computer. Mathematically, the trajectory identification process can be modelled as an inverse source problem, where one intends to identify the  trajectory of a moving point-source by measuring the corresponding wave field. We would like to emphasize that the proposed technology is novel and different from the gesture-recognition one proposed  in \cite{LWY}. In \cite{LWY}, the sensor identifies the body language by sending a point wave to interact with the human body and then collecting the reflected wave data of the passive scatter, and mathematically it is modelled as an inverse medium scattering problem.

\subsection{Mathematical modeling}

We present the mathematical formulation for the proposed input/instruction technology. Let the emitter generate a time-dependent point scalar wave of the following form
\begin{equation}\label{eq:point wave}
F_{\omega_0}(x,t; z_0):=\lambda(t) \delta(x-z_0), \quad (x, t)\in\mathbb{R}^3\times\mathbb{R}_+,
\end{equation}
where $x$ and $t$ denote, respectively, the spatial and temporal variables. In \eqref{eq:point wave}, $\delta$ is the Dirac delta distribution in space, $z_0\in\mathbb{R}^3$ signifies the position of the point source and $\lambda(t): \mathbb{R}\rightarrow \mathbb{R}$ is a casual function satisfying $\lambda(t)=0$ for $t<0$. Throughout, we take
\begin{equation}\label{eq:tthhh}
\lambda(t):=\sin(\omega_0 t)\quad \mbox{for}\ t>0,
\end{equation}
to be a time-harmonic signal, where $\omega_0\in\mathbb{R}_+$ denotes the frequency of the point wave. It is remarked that the frequency $\omega_0$ is critical in our study and shall be properly chosen in what follows. The use of the time-harmonic wave with a single frequency is important from a practical viewpoint. Indeed, it enables us to distinguish the identification signals from the various possible background noise. Assume that the wave speed in the homogeneous background space is $c_0\in\mathbb{R}_+$. Let $\Omega$ and $D$ be two bounded domains in $\mathbb{R}^3$ such that $\Omega\subset D$. Here $\Omega$ models the human body, who performs the input/instruction using the emitter, whereas $D$ models the region containing the motion of the emitter. Let $c(x)$, $x\in\Omega$, signify the wave speed in the human body, and in what follows, we extend $c$ to the whole space by setting $c(x)=c_0$ for $x\in\mathbb{R}^3\backslash\Omega$. Let $u(x, t)$ denote the wave field in the space, which satisfies the following PDE system  for $u\in H^1(\mathbb{R}_+, H_{loc}^2(\mathbb{R}^3))$,
\begin{equation}\label{eq:we1}
\begin{cases}
&\displaystyle{ \frac{1}{c^2(x)} \frac{\partial^2 u}{\partial t^2}(x, t)-\Delta u(x, t)=\lambda(t) \delta(x-z_0)},\quad (x, t)\in\mathbb{R}^3\times\mathbb{R}_+,\medskip\\
& u(x, 0)=\partial_t u(x, 0)=0.
\end{cases}
\end{equation}
If the point source is moving, one should replace $z_0$ in \eqref{eq:point wave} by $z_0(t): [0, T]\rightarrow \mathbb{R}^3$. $z_0(t)$ signifies the trajectory of the moving point source and $T$ represents the terminal time of the motion. It is assumed that  $z_0(t)\in C^1[0,T]$. Throughout the rest of our study, we also assume that the emitter stops emitting the  wave after the terminal time $T$. Let $\Gamma$ be an open surface in $\mathbb{R}^3$ outside $D$, which denotes the measurement surface. Define
\begin{equation}\label{eq:ms1}
\Lambda_{\omega_0}(z_0; T)=u(x, t)\big|_{(x, t)\in \Gamma\times [0, T]}.
\end{equation}
The inverse problem that we are concerned with is to recover $z_0(t)$ by a knowledge of the measurement operator $\Lambda_{\omega_0}(z_0; T)$, namely,
\begin{equation}\label{eq:ip1}
\Lambda_{\omega_0}(z_0; T)\longrightarrow z_0([0,T]).
\end{equation}

We would like to emphasize that the reconstruction of the trajectory in \eqref{eq:ip1} should be independent of the inhomogeneity of the background wave speed, namely $c$. This is associated with the scenario that the input/instruction recognition should be independent of the human performing that input/instruction. The strategy that we shall develop in what follows to tackle this challenging issue is to properly choose the frequency $\omega_0$ such that the wave scattering generated from the inhomogeneous medium is nearly negligible compared to that generated from the point wave. Furthermore, we shall develop a direct reconstruction algorithm without any matrix inversion, so that the method is robust with respect to the measurement noise. In doing so, one can disregard the wave scattering from the inhomogeneity, treating it as the measurement noise. Another salient feature of our study is that the measurement data used for the recovery. For the practical consideration, we only make use of the measurement of the wave field in a limited aperture. Those features distinguish our present study from the existing ones in the literature.

We are aware of the work \cite{24Naka} on recovering the motion of a moving point wave, which is much related to the present study. However, the method therein may not fit in the input/instruction framework considered in the current article. On the one hand, the authors of \cite{24Naka} only considered the reconstruction for point source located in a uniform and homogeneous background space. On the other hand, the measurement data in \cite{24Naka} are observed at a single point but it requires information of $u$, $\partial_t u$, $\partial_j u$, $\partial_t^2 u$, $\partial_t\partial_j u$, and $\partial_i\partial_j u\,\,(i,j=1,2,3)$, where $\partial_j u=\partial u/\partial x^j$ for $x=(x^j)_{j=1}^3$. Furthermore, it is required that $u$ is always non-vanishing at the observation point. Finally, we would like to refer to \cite{1Mon,2BaoG,4ArridgeSR,Bao,ClaKli,10Fokas,12Isakov,Isa2,Kli1,LU,Yam2,SU} and the references therein for  relevant studies and discussions on practical applications on the inverse source problems.

\section{Mathematical Analysis and Numerical Recovery}

In this section, we develop the recovery method for the inverse problem \eqref{eq:ip1} associated with the novel input/instruction approach. We first study the influence of the background inhomogeneity $(\Omega; c)$, and then we propose and rigorously justify the reconstruction scheme.

\subsection{Wave scattering from the background inhomogeneity}

We introduce the following PDE system for $u_0(x, t)\in H^1(\mathbb{R}_+, H_{loc}^2(\mathbb{R}^3))$,
\begin{equation}\label{eq:we2}
\begin{cases}
&\displaystyle{\frac{1}{c_0^2} \frac{\partial^2 u_0}{\partial t^2}(x, t)-\Delta u_0(x, t)=\lambda(t) \delta(x-z_0)},\quad (x, t)\in\mathbb{R}^3\times\mathbb{R}_+,\medskip\\
& u_0(x, 0)=\partial_t u_0 (x, 0)=0.
\end{cases}
\end{equation}
\eqref{eq:we2} describes the wave propagation when the point source is located in the homogeneous space with the background wave speed to be $c_0$. Then, we have

\begin{thm}\label{thm:main1}
Let $u$ and $u_0$ be, respectively, the solutions to \eqref{eq:we1} and \eqref{eq:we2}, with $z_0\in D$ being any fixed point in $D$. Suppose that $\omega_0\in\mathbb{R}_+$ is chosen such that
\begin{equation}\label{eq:cond1}
\omega_0^2\|c_0^{-2}-c^{-2}\|_{L^\infty(\mathbb{R}^3)}\leq \varepsilon\ll 1.
\end{equation}
Then there holds
\begin{equation}\label{eq:cond2}
\|u(x, t)-u_0(x, t)\|_{L^\infty(\Gamma\times [0, T])}\leq C \varepsilon,
\end{equation}
where $C$ is a positive constant depending only on $\|c\|_{L^\infty(\Omega)}, c_0, \omega_0$ and $D, \Gamma$.
\end{thm}

\begin{proof}
We reduce the wave scattering problems \eqref{eq:we1} and \eqref{eq:we2} to the frequency regime by introducing
\begin{equation}\label{eq:fq1}
u(x, t)=\sin(\omega_0 t) \Re \hat{u}(x),\ \ \ u_0(x, t)=\sin(\omega_0 t)  \Re\hat{u}_0(x)\quad\mbox{for}\ \ t>0.
\end{equation}
It is easily verified that
\begin{equation}\label{eq:rd1}
-\Delta \hat u-\frac{\omega_0^2}{c^2} \hat u=\delta(x-z_0),
\end{equation}
and
\begin{equation}\label{eq:rd2}
-\Delta \hat u_0-\frac{\omega_0^2}{c_0^2} \hat u_0=\delta(x-z_0),
\end{equation}
For both $\hat u$ and $\hat u_0$, we need impose the classical Sommerfeld radiation condition (cf. \cite{CK})
\begin{equation}\label{eq:radiation}
\lim_{\|x\|\rightarrow +\infty} \|x\|\Big\{ \frac{\partial p(x)}{\partial \|x\|}-\mathrm{i} k_0 p(x)  \Big\}=0,
\end{equation}
where $k_0=\omega_0/c_0$ and $p= \hat u$ or $\hat u_0$. The solution to \eqref{eq:rd2} is given by
\begin{equation}\label{eq:fund1}
\hat{u}_0=\Phi_{k_0}(x, z_0):=\frac{1}{4\pi} \frac{e^{\mathrm{i}k_0\|x-z_0\|}}{\|x-z_0\|}.
\end{equation}
The solution to \eqref{eq:rd1} is implicitly given in the following integral equation,
\begin{equation}\label{eq:rd12}
\hat u(x)=\Phi_{k_0}(x, z_0)+\int_{\mathbb{R}^3}\left[\left(\frac{\omega_0^2}{c^2}-\frac{\omega_0^2}{c_0^2} \right) \hat u\right](y) \Phi_{k_0}(x, y)\ \mathrm{d}V(y).
\end{equation}
Introducing the linear integral operator
\begin{equation}\label{eq:int1}
K[\hat u](x):=\int_{\mathbb{R}^3}\left[\left(\frac{\omega_0^2}{c^2}-\frac{\omega_0^2}{c_0^2} \right) \hat u\right](y) \Phi_{k_0}(x, y)\ \mathrm{d}V(y),
\end{equation}
\eqref{eq:rd12} can be written as the following operator equation
\begin{equation}\label{eq:int2}
(I-K)[\hat u](x)=\Phi_{k_0}(x, z_0).
\end{equation}
By \eqref{eq:cond1}, we see that
\begin{equation}\label{eq:int3}
\|K\|_{\mathcal{L}(L^2(\Omega), L^2(\Omega))}\leq C \varepsilon,
\end{equation}
where $C$ is a positive constant depending only on $\|c\|_{L^\infty(\Omega)}, c_0$ and $\omega_0$. By combining \eqref{eq:int2} and \eqref{eq:int3}, one can readily show that
\begin{equation}\label{eq:int4}
\|\hat u\|_{L^2(\Omega)}\leq C\varepsilon.
\end{equation}
Finally, by inserting \eqref{eq:int4} into \eqref{eq:rd12}, one can easily show that
\[
\|\hat u-\hat u_0\|_{L^\infty(\Gamma)}\leq C\varepsilon,
\]
which in turn readily implies \eqref{eq:cond2}.

The proof is complete.
\end{proof}

\subsection{Trajectory recovery}

In this subsection, we consider the scenario that the point wave is in motion following a path $z_0(t): [0, T]\rightarrow \mathbb{R}^3$. In what follows, we set
\begin{equation}\label{eq:vel}
v(t):=\frac{\mathrm{d}z_0(t)}{\mathrm{d} t}\quad \mbox{for}\ \ t\in (0, T),
\end{equation}
which denotes the velocity of the moving emitter at time $t$. Throughout the rest of the paper, we assume that
\[
\|v(t)\|\leq c_0\quad \mbox{for}\ \ t\in (0, T).
\]
For $t\in (0, T)$ and $x\in\Gamma$, define the retarded time $\tau$ by the positive solution to 
\begin{equation}\label{eq:ret}
\tau=t-\frac{\|x-z_0(\tau)\|}{c_0}.
\end{equation}

\begin{thm}\label{thm2.2}
	Let $u_0(x,t)$ be the solution to \eqref{eq:we2}. Suppose that there exists a function $\varepsilon(t)$ satisfying $0<\varepsilon(t)\ll 1, \forall t\in (0, T)$ and
	\begin{equation}\label{eq:conn1}
	t-\tau=\frac{\|x-z_0(\tau)\|}{c_0}=\varepsilon(t)\frac{2\pi}{\omega_0}, \quad (x, t)\in \Gamma\times (0, T).
	\end{equation}
	where $\tau$ is the retarded time defined in \eqref{eq:ret}. Then it holds that
	\begin{equation}\label{eq:appr1}
	\displaystyle u_0(x,t)=\frac{\sin \omega_0 t }{4\pi\|x-z_0(t)\|}+\mathcal{O}(\varepsilon(t)),  \quad (x, t)\in \Gamma\times (0, T).
	\end{equation}
\end{thm}

\begin{proof}
Then the solution $u_0(x, t)$ to \eqref{eq:we2} is given by the retarded potential (cf. \cite{24Naka})
\begin{equation}\label{eq:1.2}
u_0(x,t)=\frac{\sin \omega_0\tau }{4\pi\|x-z_0(\tau)\|(1-c_0^{-1}\langle x-z_0(\tau),v(\tau) \,\rangle\,)},
\end{equation}
where $\tau$ is the retarded time defined in \eqref{eq:ret} and $\langle\cdot,\cdot\rangle$ denote the usual $L^2$ inner product with respect to the spatial variable.

By straightforward calculations, we have
\begin{equation}\label{eq:dd1}
\begin{split}
  \sin\omega_0\tau=&\sin\omega_0 \left(t-\varepsilon(t)\frac{2\pi}{\omega_0}\right)\\
                 =& \sin\omega_0 t \cos(2\pi\varepsilon(t))-\cos\omega_0 t \sin(2\pi\varepsilon(t))\\
                  =& \sin\omega_0 t +\mathcal{O}(\varepsilon(t)).\\
\end{split}
\end{equation}

Since $t-\tau=\|x-z_0(\tau)\|/c_0=\varepsilon(t)2\pi/\omega_0$, one can also get
\begin{equation}\label{eq:dd2}
\begin{split}
  \|x-z_0(\tau)\|=&\left\|x-z_0\left(t-\varepsilon(t)\frac{2\pi}{\omega_0}\right)\right\|\\
               =&\left\|x-z_0(t)+z'_0(\eta)\varepsilon(t)\frac{2\pi}{\omega_0}\right\|\\
               =& \|x-z_0(t)\||1+\mathcal{O}(\varepsilon(t))|,\\
\end{split}
\end{equation}
where $t-\varepsilon(t)2\pi/\omega_0<\eta\leq t$, and
\begin{equation}\label{eq:dd3}
\begin{split}
  1-\frac{\langle x-z_0(\tau),v(\tau)\,\rangle}{c_0}
  =&1-\frac{\|x-z_0(\tau)\|\,\|v(\tau)\|\cos\beta}{c_0}\\
  =&1-\varepsilon(t)\frac{2\pi}{\omega_0}\,\|v(\tau)\|\cos\beta\\
  =&1-\mathcal{O}(\varepsilon(t)),
\end{split}
\end{equation}
where $\beta\in [0, 2\pi]$ denotes the angle between $x-z_0(\tau)$ and $v(\tau)$. Finally, by plugging \eqref{eq:dd1}, \eqref{eq:dd2} and \eqref{eq:dd3} into \eqref{eq:1.2}, along with straightforward asymptotic analysis, one can obtain
\begin{equation*}
 \displaystyle u_0(x,t)=\frac{\sin \omega_0 t }{4\pi\|x-z_0(t)\|}+\mathcal{O}(\varepsilon),  \quad (x, t)\in \Gamma\times (0, T).
\end{equation*}

The proof is complete.

\end{proof}

\begin{rem}\label{rem:1}
Theorem~\ref{thm2.2} is critical for our subsequent development of the reconstruction algorithm for the inverse problem \eqref{eq:ip1}. It is remarked that the sufficient condition \eqref{eq:conn1} for Theorem~\ref{thm2.2} can be fulfilled by properly choosing a low frequency $\omega_0$ of the point wave, as well as by requiring that the distance between the motion of the point wave and the measurement surface $\Gamma$ is within a reasonable range. In the setup of the proposed input/instruction technology, this is equivalent to saying that the emitter generates a point wave with a relatively low frequency, and the person who performs the input/instruction should stand within a reasonable distance from the sensor. If the sensor is installed on the computer, the latter condition means that the human should not be very far away from the computer.
\end{rem}

We are now in a position to present the imaging functional for the inverse problem \eqref{eq:ip1}, which can be used to qualitatively determine the  trajectory $z_0([0, T])$ by knowledge of $\Lambda_{\omega_0}(z_0; T)$. Define
\begin{equation}\label{eq:test}
\phi (x,t; z)=\frac{\sin \omega t }{4\pi \|x-z\|}\quad\mbox{for}\ \ (x, t, z)\in \Gamma\times (0, T)\times D,
\end{equation}
and
\begin{equation}\label{eq:ind1}
I(t, z)=\frac{\Big|\big\langle u(x, t; z_0(t)), \phi(x, t; z) \big\rangle_{L^2(\Gamma)}\Big|}{\|u(x, t; z_0(t))\|_{L^2(\Gamma)}\|\phi(x, t; z)\|_{L^2(\Gamma)}}\quad\mbox{for}\ \ (z, t)\in D\times (0, T),
\end{equation}
where $u(x, t; z_0(t))=u(x, t)$. Then, we have
\begin{thm}\label{thm:main3}
Let $u(x, t)$ be the measurement data for $(x, t)\in \Gamma\times (0, T)$, corresponding to a moving point wave described in \eqref{eq:we1} and let $I(t, z)$ be defined in \eqref{eq:ind1}. Suppose that $\omega_0$ and $D, \Gamma$ are such chosen that both \eqref{eq:cond1} and \eqref{eq:conn1} are fulfilled. Then, for any fixed $t\in (0, T)$,
\begin{equation}
I(t, z)\approx I_0(t, z):=\frac{\Big|\big\langle \widetilde u_0 (x, t; z_0(t)), \phi(x, t; z) \big\rangle_{L^2(\Gamma)}\Big|}{\|\widetilde u_0(x, t; z_0(t))\|_{L^2(\Gamma)}\|\phi(x, t; z)\|_{L^2(\Gamma)}}\quad\mbox{for}\ \ (z, t)\in D\times (0, T),
\end{equation}
where $\widetilde u_0(x, t; z_0(t))=\widetilde u_0(x, t)$ is defined in \eqref{eq:appr1}. Moreover, $I_0(t, z)$ achieves its maximum value when $z=z_0(t)$ with $I_0(t, z_0(t))=1$.
\end{thm}

\begin{proof}
Since both \eqref{eq:cond1} and \eqref{eq:conn1} are satisfied, by combining Theorems~\ref{thm:main1} and \ref{thm2.2}, one can easily obtain that
\[
u(x, t; z_0(t))\approx \widetilde u_0(x, t; z_0(t)).
\]
The maximum behaviour of $I_0(t, z)$ can be easily seen by using the Cauchy-Schwarz inequality.

The proof is complete.

\end{proof}

Based on Theorem~\ref{thm:main3}, we next present the trajectory reconstruction scheme for the inverse problem \eqref{eq:ip1}.

\begin{table}[h]
\centering
\begin{tabular}{cp{.8\textwidth}}
\toprule
\multicolumn{2}{l}{{\bf Algorithm T:}\quad Reconstruction of the trajectory of a moving point emitter}\\
\midrule
 {\bf Step 1} & Properly choose a low frequency $\omega_0$ and start to generate the point wave of the form \eqref{eq:point wave}--\eqref{eq:tthhh} from $t=0$.  \\
 \midrule
{\bf Step 2} & Set the emitter in motion following a specific path $z_0(t)$, depending on the desired input/instruction. The sensor collects the scattered wave data $u(x, t)$ at the measurement points $\{x_j\}\in \Gamma$ and a sequence of discrete time points $\{t_j\}\in (0, T]$. \\
\midrule
{\bf Step 3} & Select a  sampling mesh $\mathcal{T}_h$ in a region containing $D$. For each time point $t_j$, calculate the imaging functional $I(t_j, z)$ defined in \eqref{eq:ind1} for each $z\in \mathcal{T}_h$.\\
\midrule
{\bf Step 4} & Locate the global maximum point of $I(t_j, z)$ for $z\in\mathcal{T}_h$ with the maximum value being nearly $1$, which is the approximation to $z_0(t_j)$. \\
\midrule
{\bf Step 5} & After the approximate determination of $z_0(t_j)$ for all the discrete time points, one can post-process the discrete data points to obtain the reconstructed trajectory $z_0([0, T])$.  \\
\bottomrule
\end{tabular}
\end{table}

We would like to emphasize that clearly the proposed scheme yields a qualitative reconstruction of the moving trajectory. Nevertheless, we shall conduct extensive numerical experiments in Section~\ref{sect:numerical} to show that the proposed method can provide an effective and efficient reconstruction of the trajectory. Before we present the numerics, we shall develop two speed-up techniques in the next subsection, which can significantly reduce the computational cost in Steps 3 and 4 of Algorithm T. The post-processed technique shall be discussed in the next section.

\subsection{Two speed-up techniques}\label{sect:FastTechniques}

Motivated by the fact that the sound speed is finite, we develop two local sampling techniques to reduce the computational efforts.
The rationale behind  is to replace the global sampling domain $D$ by a sequence of sampling subdomains, e.g., balls. We will specify the details of constructing the dynamic sampling domains in sequence and in parallel, respectively. Without loss of generality, assume that $(0, T]$ is uniformly divided into $N_t$ time steps $t_j=jT/N_t, \,j=1,2,\cdots,N_t,\, N_t\in\mathbb{N}$.

\subsubsection{Sequential tuning}

First, we present a time-marching technique to recover the locations of the source at different time steps.

\begin{enumerate}
  \item Locate the global maximum point of $I(t_1,z)$ for $z\in D$ and denote it by $z_1$, i.e.,
	      \[ z_1:=\arg\max\limits_{z\in D}I(t_1,z). \]
  \item Narrow down the sampling domain  to $B_1 \subset D$, where $B_1$ is a ball centered at $z_1$ with radius $\|v(t)\|_{\infty} T/N_t$. Then locate the local maximum point of $I(t_2,z)$ for $z\in B_1$ and denote it by $z_2$, i.e.,
	  \[ z_2:=\arg\max\limits_{z\in B_1}I(t_2,z). \]
	It is obvious that $z_2$ is the global  maximum point of $I(t_2,z)$ in $D$.
  \item Following this idea and progressing with the time sequence $\{t_j\}$, one can  construct a sequence of local sampling subdomains $\{B_j\}$ and find the corresponding local maximum points $\{z_j\}$ in an alternating manner:
	 \[z_1\to B_1\to z_2\to B_2\to z_3\to B_3\to \cdots\]
	 such that
	\begin{align*}
	 B_j:=&\{x\in\mathbb{R}^3\mid \|x-z_j\|<\|v(t)\|_{\infty} T/N_t\}, \\
      z_j:=&\arg\max\limits_{z\in B_{j-1}}I(t_j,z),\quad j=2,3,\cdots,N_t.
	   \end{align*}
	Thus the ordered sequence of $\{z_j\},\, j=1,2,\cdots,N_t$ forms a discrete marching trajectory.
\end{enumerate}

\subsubsection{Parallel tuning}
Next,  we propose a hybrid method that combines dichotomy with parallel computing techniques.
Let $\lceil \cdot \rceil$ and $ \lfloor \cdot \rfloor$ denote the rounding functions which round $\cdot$ to the nearest integer greater and less than $\cdot$, respectively.
\begin{enumerate}
  \item Locate the global maximum point of $I(t_{N_t},z)$ for $z\in D$ and denote it by $z_{01}$, i.e.,
	   \[ z_{01}:=\arg\max\limits_{z\in D}I(t_{N_t},z). \]
  \item Narrow down the sampling domain to $B_{01}\subset D$, where $B_{01}$ is a ball centered at $z_{01}$ with radius $\|v(t)\|_{\infty} \lceil \frac{N_t}{2}\rceil T/N_t$.  Then locate the local maximum point of $I(t_{\lfloor \frac{N_t}{2}\rfloor},z)$ for $z\in B_{01}$ and denote it by $z_{11}$, i.e.,
	  \[ z_{11}:=\arg\max\limits_{z\in B_{01}}I(t_{\lfloor \frac{N_t}{2}\rfloor},z). \]
	It is obvious that $z_{11}$ is the global maximum point of $ I(t_{\lfloor \frac{N_t}{2}\rfloor},z)$ in $D$.
  \item Employ the parallel computing  to simultaneously locate the maximum points of $I(t_{\lfloor \frac{N_t}{2}\rfloor},z)$ and $I(t_{\lfloor \frac{3N_t}{2^2}\rfloor},z)$ in $B_{11}$, respectively, where $B_{11}$ is a ball  centered at $z_{11}$ with  radius $|v(t)|_{\infty} \lceil \frac{N_t}{2^2}\rceil T/N_t$. The corresponding maximum points can be denoted  by $z_{21}$ and $z_{22}$, respectively, i.e.,
	\begin{align*}
	z_{21}:=\arg\max\limits_{z\in B_{11}}I(t_{\lfloor \frac{N_t}{2^2}\rfloor},z), \\
	z_{22}:=\arg\max\limits_{z\in B_{11}}I(t_{\lfloor \frac{3N_t}{2^2}\rfloor},z).
	\end{align*}
	
  \item Following this idea, in general one can construct the sequence of local sampling domains $\{B_{in}\}$ and find the corresponding local maximum points $\{z_{in}\}$ in an alternating and parallel manner:
	\[z_{01}\to B_{01}\to z_{11}\to B_{11}\to
	\begin{cases}
	z_{21}\to B_{21}\to \begin{cases}
	z_{31}\to B_{31}\to\cdots\\
	z_{32}\to B_{32}\to\cdots
	\end{cases}\vspace{1mm}\\
	z_{22}\to B_{22}\to\begin{cases}
	z_{33}\to B_{33}\to\cdots\\
	z_{34}\to B_{34}\to\cdots
	\end{cases}
	\end{cases} \]
	such that
	\begin{align*}
    z_{in}:=&\arg\max\limits_{z\in B_{i-1,\lceil \frac{n}{2}\rceil}}I(t_{\lfloor \frac{(2n-1)N_t}{2^i}\rfloor},z),\\
	B_{in}:=&\{x\in\mathbb{R}^3\mid \|x-z_{in}\|<|v(t)|_{\infty} \lceil {N_t}/{2^{i+1}}\rceil T/N_t\},\\
	& n=1, \cdots, 2^{i-1} ,\, i=1, \cdots,\lfloor \log_2 N_t \rfloor.
	\end{align*}
   	Thus the sequence of $\{z_{j}\}$, $j=\lfloor \frac{(2n-1)N_t}{2^i}\rfloor$, $i=1, 2, \cdots, \lfloor \log_2 N_t \rfloor$, $n=1,2, \cdots, 2^{i-1}$, forms a discrete trajectory.
\end{enumerate}
 It is noted  that $j=1,2, \cdots ,N_t,$ if and only if $\log_2 N_t \in \mathbb{N}$.

\begin{figure}
    \hfill\subfigure[]{\includegraphics[  width=0.5\textwidth]{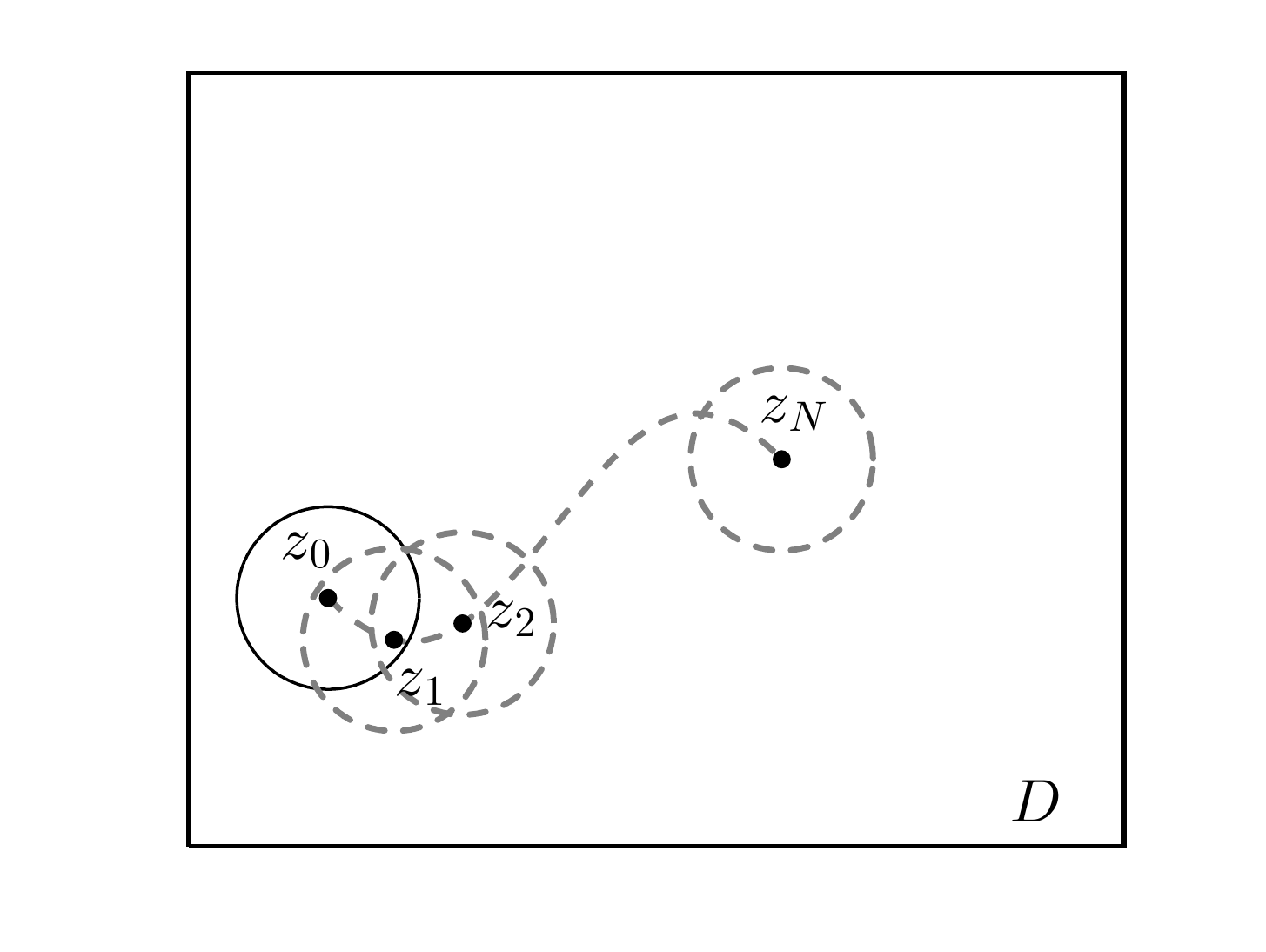}}\hfill
    \hfill\subfigure[]{\includegraphics[  width=0.5\textwidth]{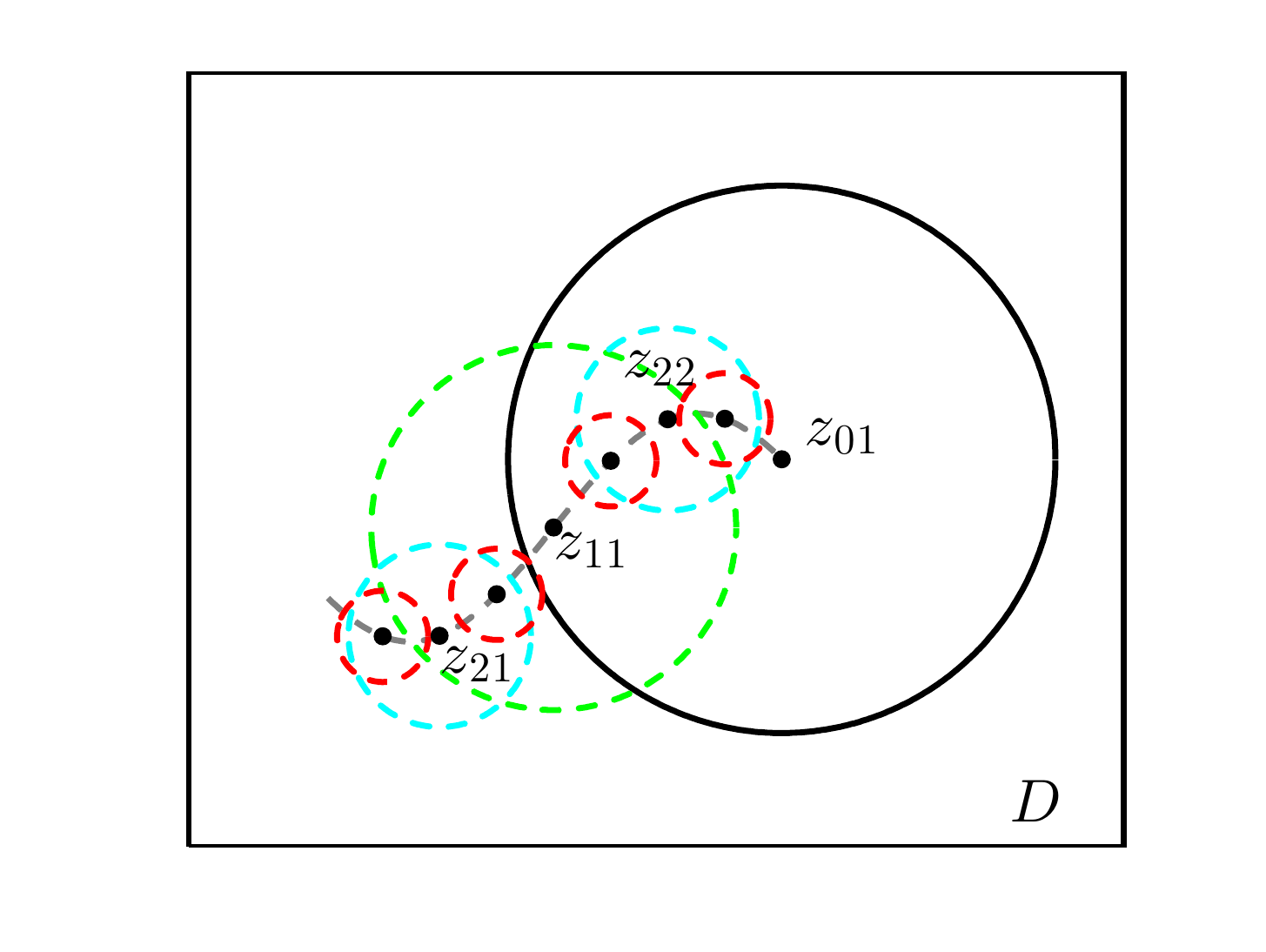}}\hfill
  \caption{Illustration of two speed-up techniques. Left: sequential tuning (the solid circle denotes $B_1$ and the dashed circles stand for $B_2$, $B_3$ and $B_{N}$); Right: parallel tuning (the black solid circle denotes $B_{01}$, the green dashed circle denotes $B_{11}$ , the cyan dashed circles denote $B_{21}$ and $B_{22}$, and the red dashed circles denote $B_{31}$, $B_{32}$, $B_{33}$ and $B_{34}$ ).}\label{fig.2}
\end{figure}

\begin{rem}
Figure \ref{fig.2} shows the idea of the two speed-up techniques.  The respective time complexity of the two techniques are
\begin{equation*}
  T_1(N_t)=\mathcal{O}(N_t),
\end{equation*}
and
\begin{equation*}
  T_2(N_t)=\mathcal{O}(1+\log_2 N_t ).
\end{equation*}
However, parallel tuning technique costs some time in assigning tasks. Moreover, time steps are not very large in our experiments. Therefore, unless otherwise specified,  we suggest using sequential tuning  technique  to reduce the  computational cost.
\end{rem}

\section{Numerical examples}\label{sect:numerical}

In this section, we will present several numerical examples to demonstrate the feasibility and effectiveness of the proposed method.

 All the following numerical experiments are carried out using  MATLAB R2013a   on a Lenovo workstation with $2.3$GHz Intel Xeon E5-2670 v3 processor and 128GB of RAM.  Synthetic  wave field data are generated by solving direct problem of \eqref{eq:we1} and \eqref{eq:we2} by using the quadratic finite elements on a truncated  spherical  domain enclosed by an absorbing boundary condition. The mesh of the forward solver is successively refined till the relative error of the successive measured wave data is below $0.1\%$. To test the stability of our proposed method, we  add random noise to the synthetic data $u(x,t)$. The noisy data are given by the following formula
\begin{equation*}
  u^\epsilon:=u(1+\epsilon r)
\end{equation*}
where $r$ are point-wise uniform random numbers, ranging from $-1$ to $1$, and $\epsilon>0$ represents the noise level.
Unless otherwise specified, we fix $\epsilon=0.05$, namely, 5\% noise was added to the measurement data.

 Next, we specify details of the discretization. Using spherical coordinates $(r, \theta, \varphi): x=(r\sin\theta\cos\varphi, r\sin\theta\sin\varphi, r\cos\theta)$,  the limited aperture  $\Gamma$  are taken on the patch of the sphere  with  radius $r=10$m, polar angle $\theta\in (\pi/4,3\pi/4)$ and azimuthal angle $\varphi\in(-\pi/4,\pi/4)$. In all subsequent experiments,  the time step is set to be $\Delta t=0.1$s, the angular frequency is set to be $\omega_0=1$rad/s and the sound speed of the background medium is chosen as $c_0=330$m/s.  Suppose that we have  $N_m$ equidistantly distributed receivers $x_m\in \Gamma,\, m=1,2,\cdots,N_m$,  and  $N_t$  equidistant recording time steps $t_j\in(0,T],\,j=1,2,\cdots,N_t$, where $N_m=200$ and $N_t=T/\Delta t$. Then  the measurement data is an $N_m\times N_t$ array whose $(m,j)$-th entry is $u_0(x_m,t_j)$ or $u(x_m,t_j)$. Hence, the discretized version of the  indicator \eqref{eq:ind1} can be written in the form
 \begin{equation*}
  \displaystyle I(t_j,z)=\frac{\sum\limits_{ m=1 }^ {N_m} {u_0(x_m,t_j) \phi(x_m,t_j;z)\Delta s_x}}{\left(\sum\limits_{ m=1 }^ {N_m} {u_0^2(x_m,t_j) \Delta s_x}\right)^{1/2} \left(\sum\limits_{ m=1 }^ {N_m} { \phi^2(x_m,t_j;z)\Delta s_x}\right)^{1/2}}, \quad j=1,2,\cdots, N_t,
 \end{equation*}
 where $\Delta s_x $ denotes  the area of the grid cell of receivers. Finally, at each time step  $t_j$, we use a uniformly distributed $100 \times 100 \times 100 $ sampling mesh $\mathcal{T}_h$  over the cube $D:=[-8\mathrm{m},\, 8\mathrm{m}]^3$ to locate the global maximum points  $\{z_j\}$ of discretized  indicator function $I(t_j,z)$ for $z\in \mathcal{T}_h$.  Thus, these discrete  points $\{z_j\}$ build the discretized version of the reconstructed  moving trajectory.

In addition, if we want to obtain a smooth curve as the reconstructed trajectory,  then we may post-process the discrete data points $\{z_j\}$ by a suitable Fourier approximation.
The post-processed trajectory  can be chosen as the following truncated Fourier expansion of order $P\in\mathbb{N}$:
 \begin{equation*}
 z(t):=a_0+\sum\limits_{n=1}^P (a_n \cos (n t) + b_n \sin (n t)), \quad t\in (0, T],
 \end{equation*}
where the Fourier coefficients are given by
 \begin{equation*}
 \begin{aligned}
   &a_0=\frac{1}{T}\sum \limits_{ j=1 }^ {N_t} z_j\frac{T}{N_t}, \\
    &a_n=\frac{2}{T}\sum \limits_{ j=1 }^ {N_t} z_j\cos (n t_j) \frac{T}{N_t}, \\
     &b_n=\frac{2}{T}\sum \limits_{ j=1 }^ {N_t} z_j\sin (n t_j) \frac{T}{N_t} ,\quad n=1,2,\cdots,P.\\
   \end{aligned}
 \end{equation*}

We would like to point out that
in order to eliminate the artifacts of large oscillations in the reconstructed trajectories, the Fourier expansion of a low order is usually preferable in the post-processing step. Conversely, if one wants to capture the presumable high-frequency details of the trajectory, then the Fourier approximation of a high order might be more suitable. Therefore, an appropriate choice of the truncated order may require some {\em a priori} information of the noise level $\epsilon$ and the smoothness of $z_0(t)$. If not otherwise specified, we will choose $P=3$ in what follows.

\begin{rem}
In order to obtain a three-dimensional visualization of the reconstructed points and curves, some 2D projections (shadows with gray color) are also added to the configurations of the reconstructions.
\end{rem}

\begin{figure}[htb!]
  \centering
   \hfill\subfigure[]{\includegraphics[  width=0.45\textwidth]{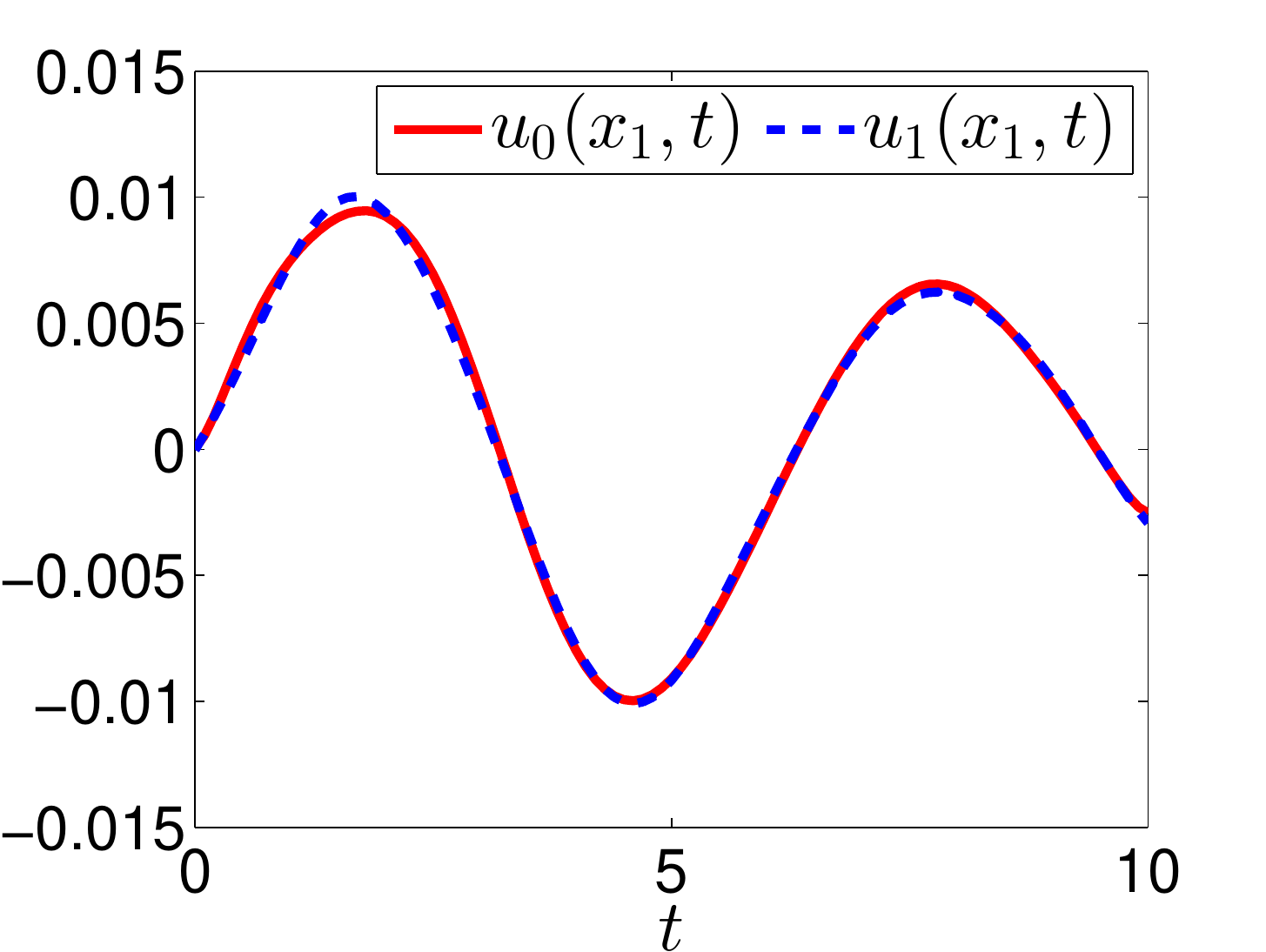}}\hfill
    \hfill\subfigure[]{\includegraphics[  width=0.45\textwidth]{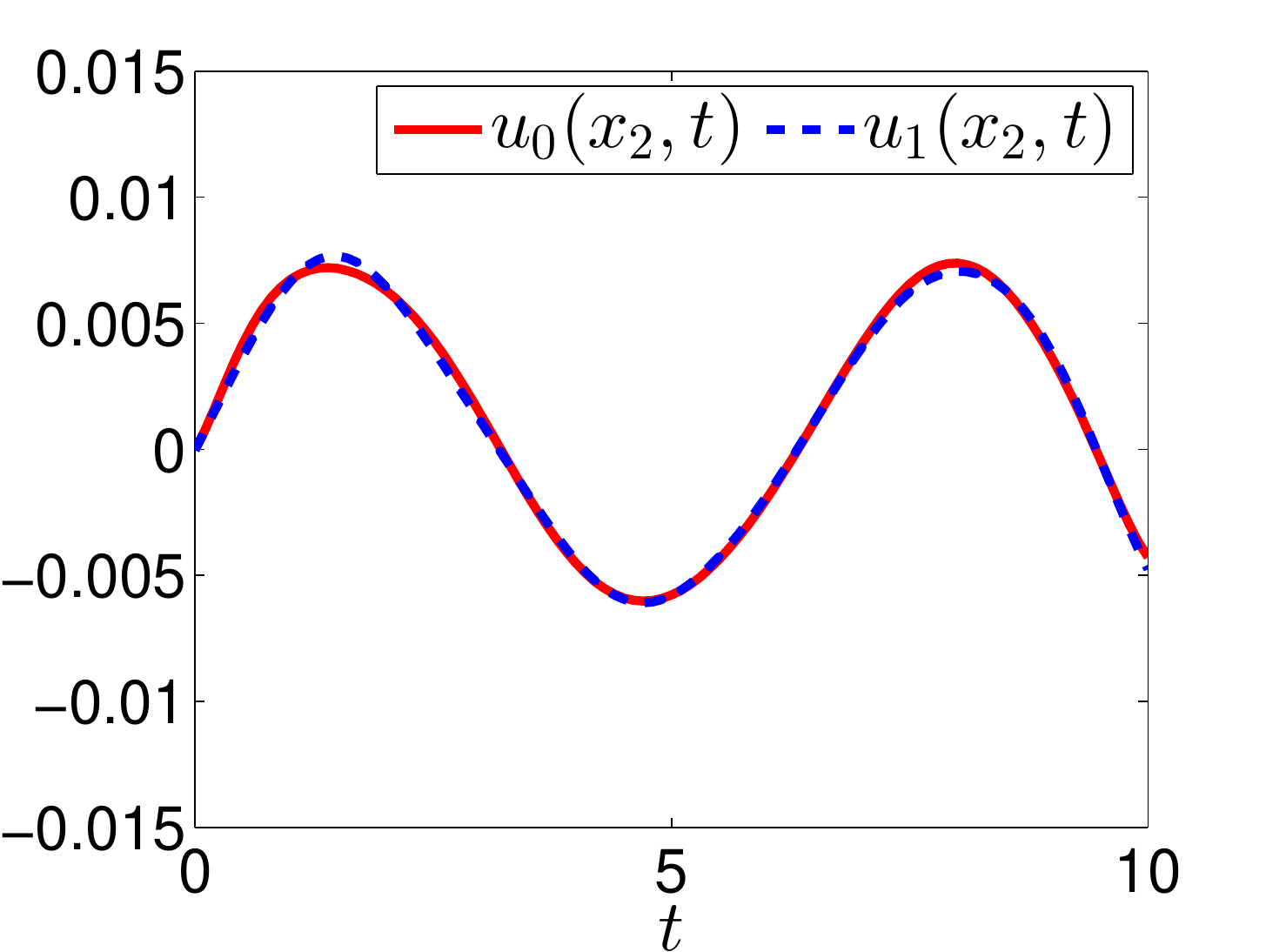}}\hfill\\
     \hfill\subfigure[]{\includegraphics[width=0.45\textwidth]{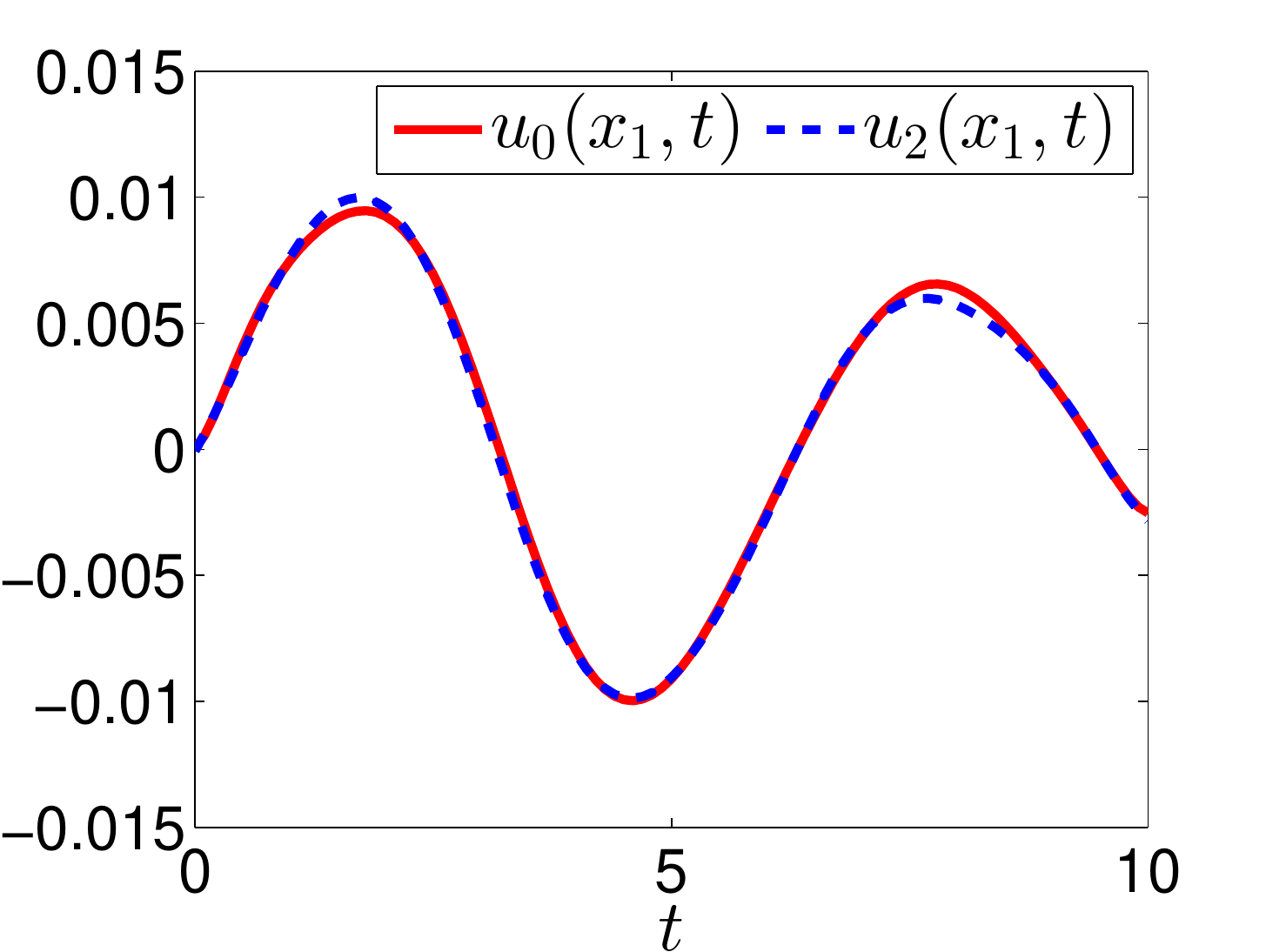}}\hfill
      \hfill\subfigure[]{\includegraphics[   width=0.45\textwidth]{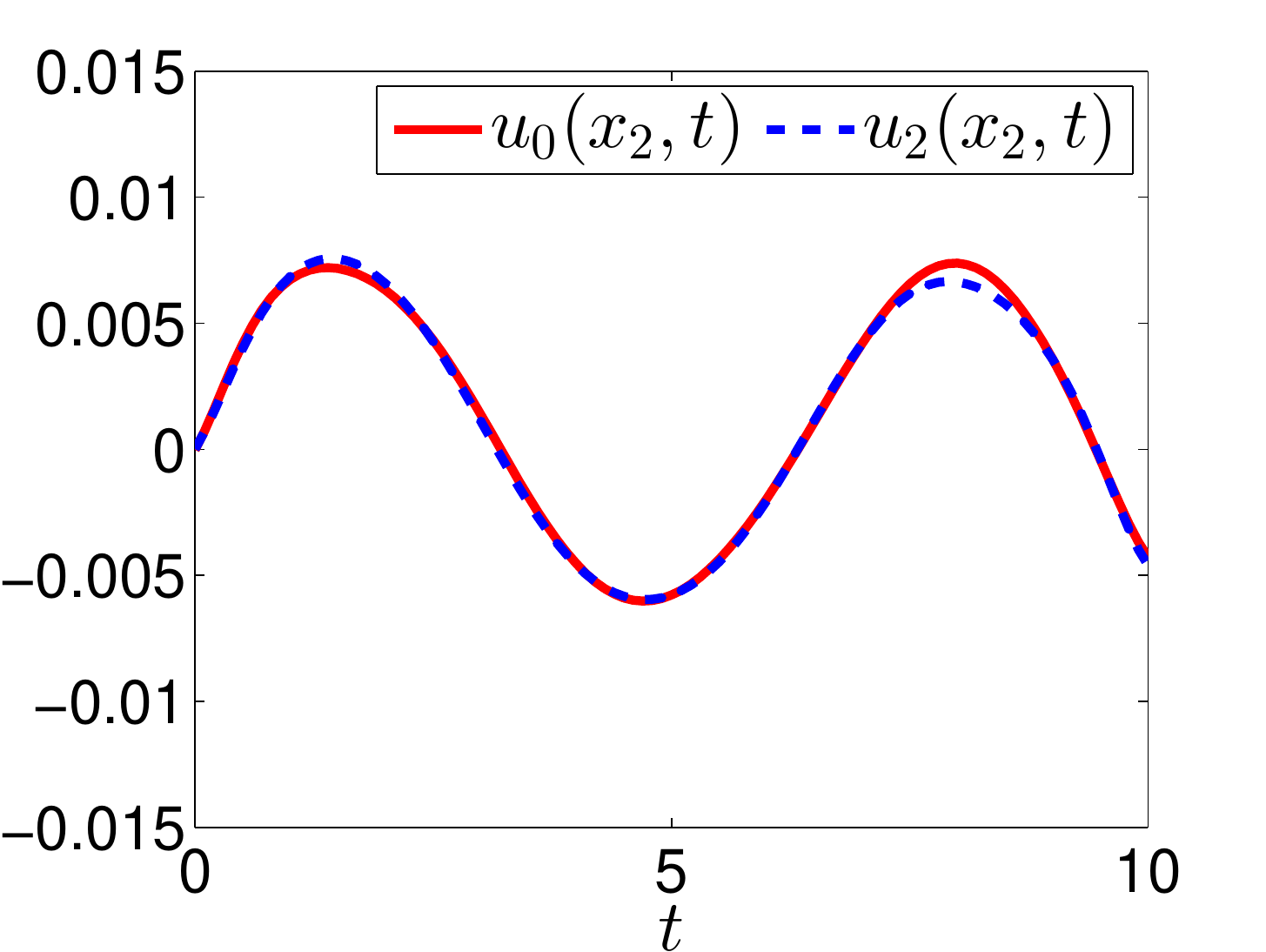}}\hfill\\
    \caption{Numerical results for Example \ref{exp0}. The top row compares Case (i) with Case (ii) at receivers $x_1$ and $x_2$, respectively; the bottom row compares Case (i) with Case (iii) at receivers $x_1$ and $x_2$, respectively.}\label{fig.3.0}
\end{figure}

\begin{rem}
In Theorem \ref{thm2.2}, the assumption $t-\tau=c_0^{-1}\|x-z_0(\tau)\|\ll 2\pi/\omega$ is only for theoretical purpose. As we will see in the  following examples, a moderately small quantity of $c_0^{-1}\|x-z_0(\tau)\|$ could be sufficient to yield satisfactory reconstruction.
\end{rem}

%%=========================example1===========================================================================

\begin{example}\label{exp0}
In the first example, we aim to verify Theorem~\ref{thm:main1} numerically, i.e., $u$ is approximated to $u_0$.
Assume that the emitter moves along the moving trajectory
\begin{equation*}
\displaystyle{z_0(t)=\left( 0, 3\cos\left(\frac{3\pi}{20}t+\frac{\pi}{4}\right),    3\sin\left(\frac{3\pi}{20}t+\frac{\pi}{4}\right) \, \right)},\quad t\in (0\mathrm{s}, 10\mathrm{s}],
\end{equation*}
which is depicted in Figure \ref{fig.3.1}(a) and served as an approximation of a letter ``C".
Next,  we  select two receivers (sensors)  which are located at $x_1=(5\mathrm{m},-5\mathrm{m}, 5\sqrt{2}\mathrm{m}) $ and $x_2=(10\mathrm{m}, 0\mathrm{m}, 0\mathrm{m})$,   and then  let both receivers (sensors)  receive wave field generated from the emitter. Here we consider  three cases.

(i)\ We consider an emitter moves in a homogeneous space with background wave speed being $c_0=330$m/s in $\mathbb{R}^3$. Here, we define the wave field by $u_0$.

(ii)\ We study a simplified scenario that a person is wearing an  emitter on one of his/her finger and moving the finger to  write a letter ``C".  We  assume that the person faces  receivers and  consider his/her body $\Omega$ as an inhomogeneity.  For simplicity, let  $c(x)=1500$m/s signify the wave speed in   $\Omega$ and $c_0=330$m/s denote wave speed  in $\mathbb{R}^3\backslash \Omega$.  We also choose  $\Omega$ as a cuboid domain with center $(-2\mathrm{m}, 0\mathrm{m}, 0\mathrm{m})$ and size $2\mathrm{m}\times10\mathrm{m}\times10\mathrm{m}$. In this case, the radiated wave field is denoted by $u_1$.

(iii)\ We assume that  the person stands back against the receivers and the wave transmits through the body to receivers. So let  $\Omega$ denote a cuboid domain with center $(2\mathrm{m}, 0\mathrm{m}, 0\mathrm{m})$ and size $2\mathrm{m}\times10\mathrm{m}\times10\mathrm{m}$. In this case, the radiated wave field is denoted by $u_2$.

The numerical results verifies that $u_0$ is approximately equal to $u_1$ and $u_2$ for both receivers  at every instant, respectively, as shown in Figure \ref{fig.3.0}. Figure \ref{fig.3.0} shows that the location of inhomogeneous medium has not made much  impact on the  measured wave field. In other words, no matter which direction we face, the received signals are almost the same.  Moreover, we also did some experiments with $\omega_0< 1$ and some other receivers. All the numerical results showed that the errors between u0 and u are very small and  the error decreases as the angular frequency decreases. However, the reconstruction of the trajectory would deteriorate as the angular frequency tends to zero.  Hence the trade-off deserves further investigations in practical applications.

In the next example, we will  recover the moving trajectories  in  the above three cases.
\end{example}

%%==================================example2=================================================================

\begin{figure}[htb!]
\subfigure[]{\includegraphics[clip,width=0.32\textwidth]{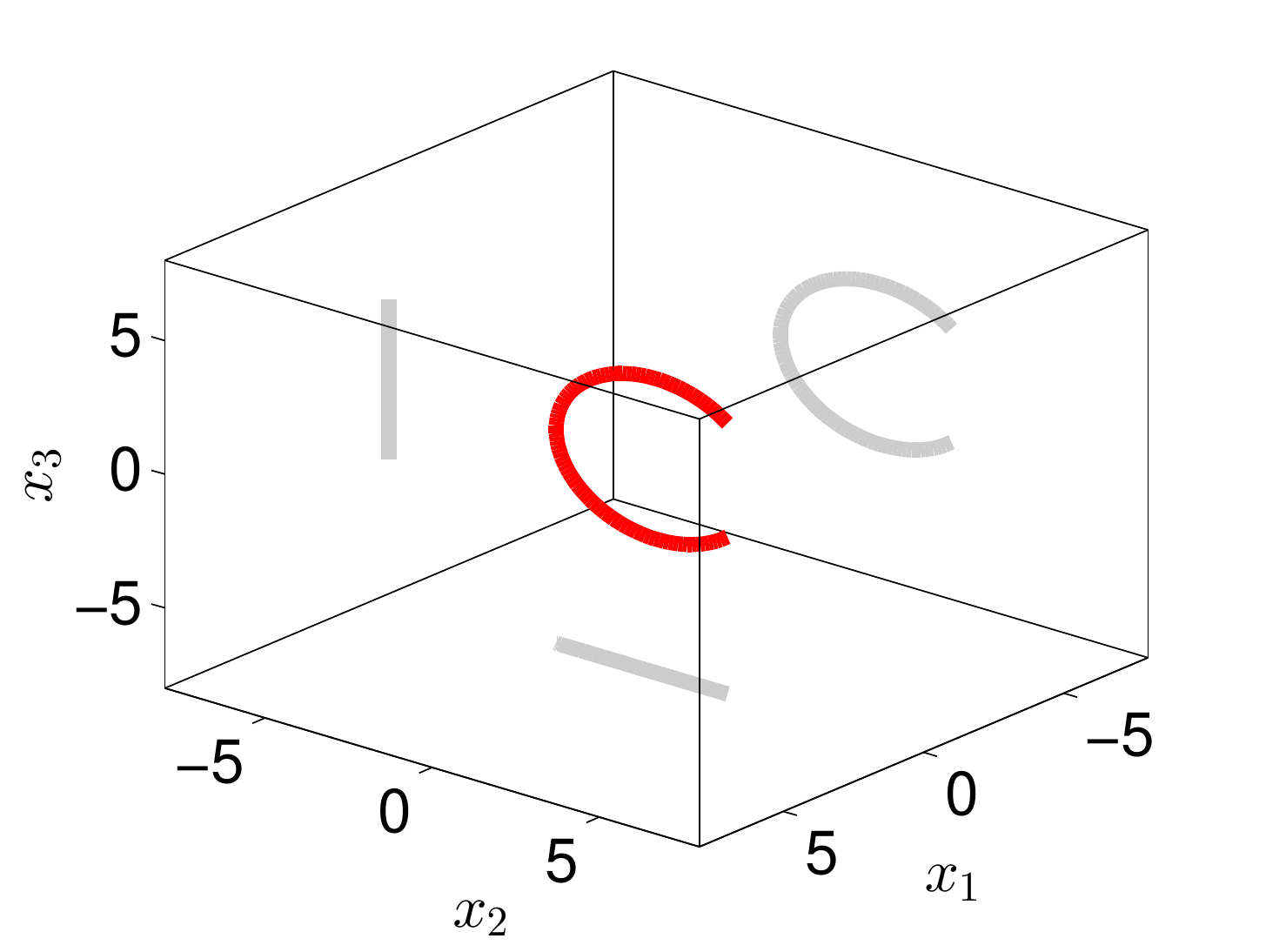}}
\subfigure[]{\includegraphics[width=0.32\textwidth]{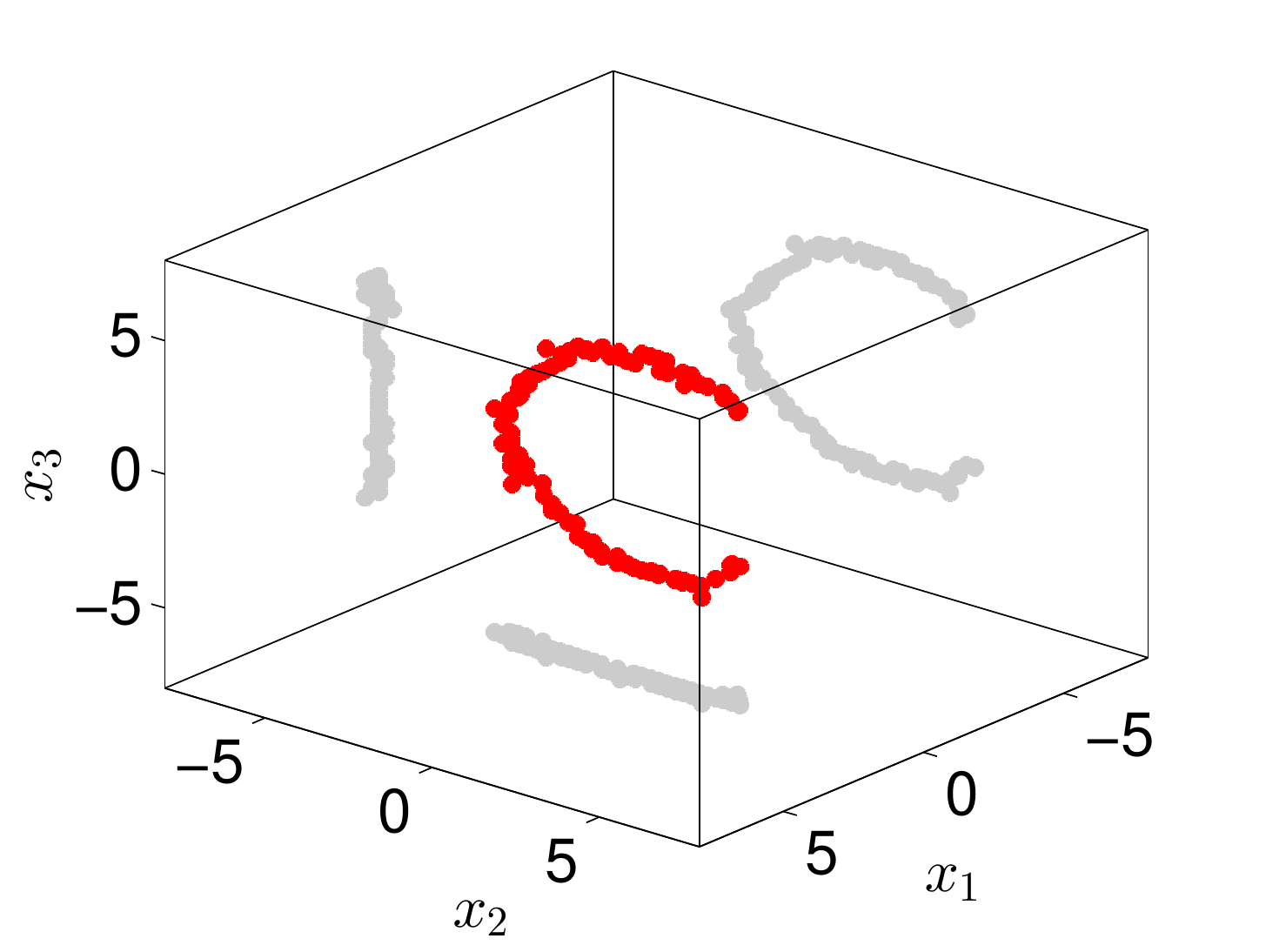}} \subfigure[]{\includegraphics[width=0.32\textwidth]{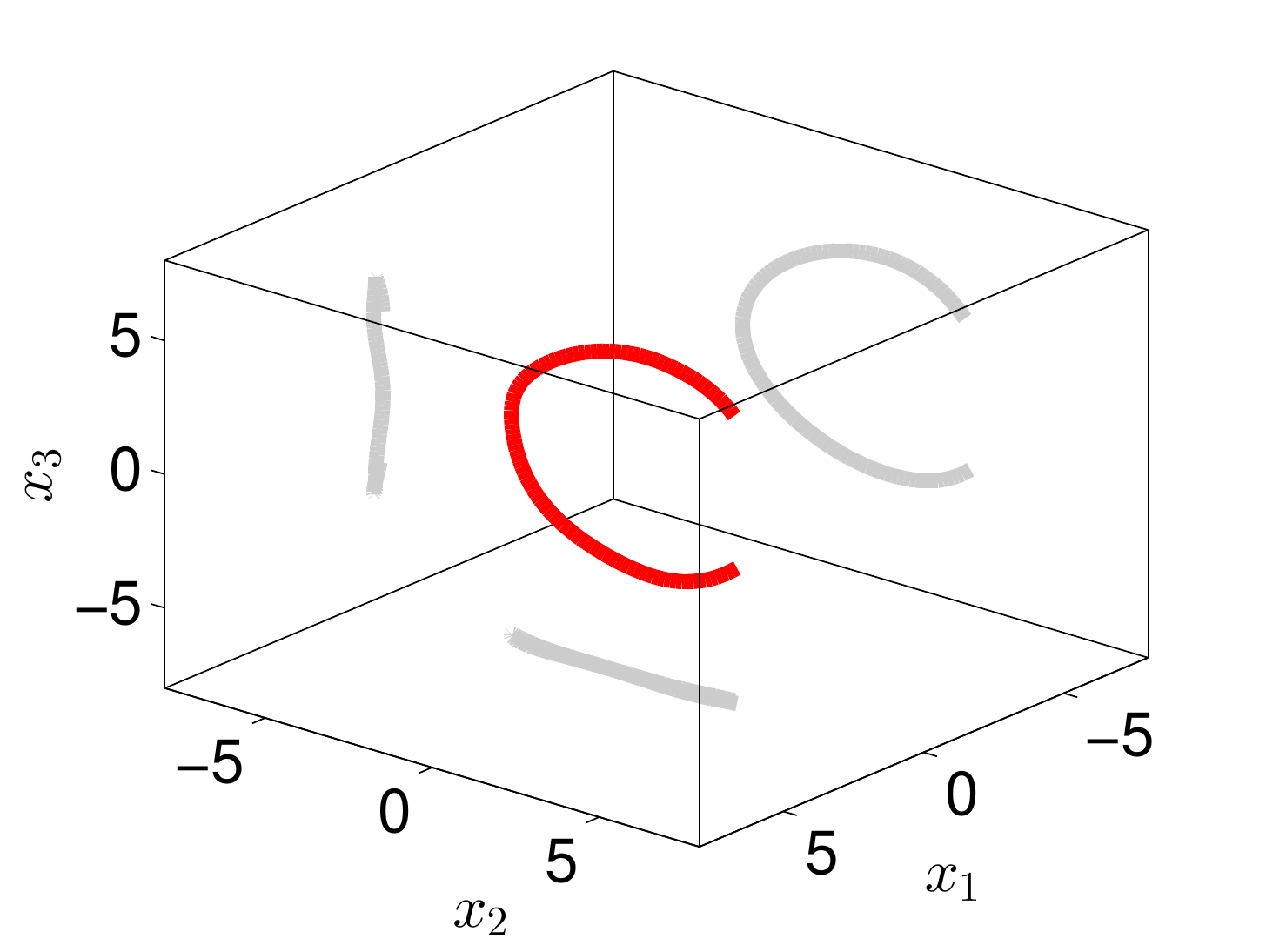}}\\
\subfigure[]{\includegraphics[width=0.32\textwidth]{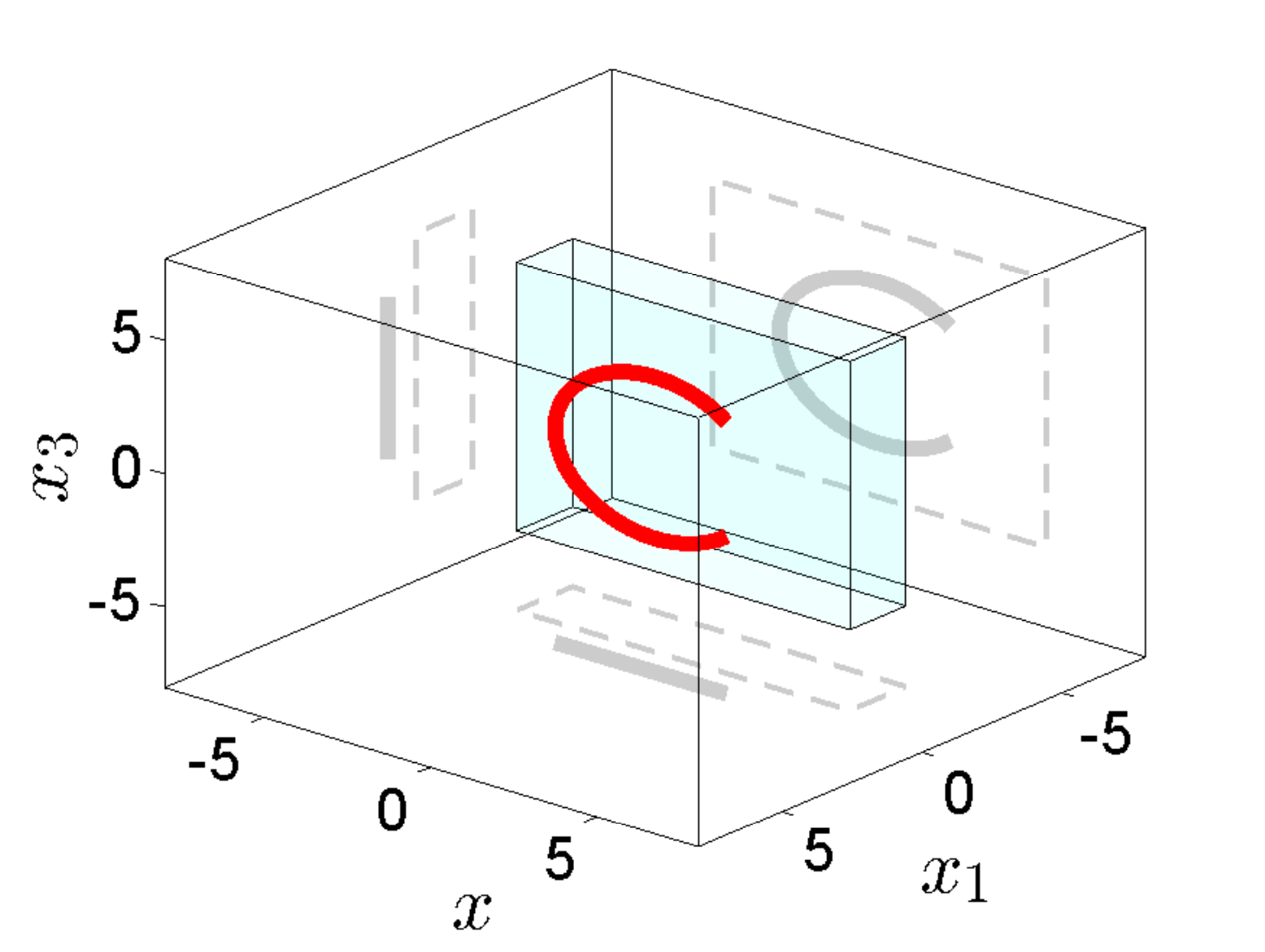}}
\subfigure[]{\includegraphics[width=0.32\textwidth]{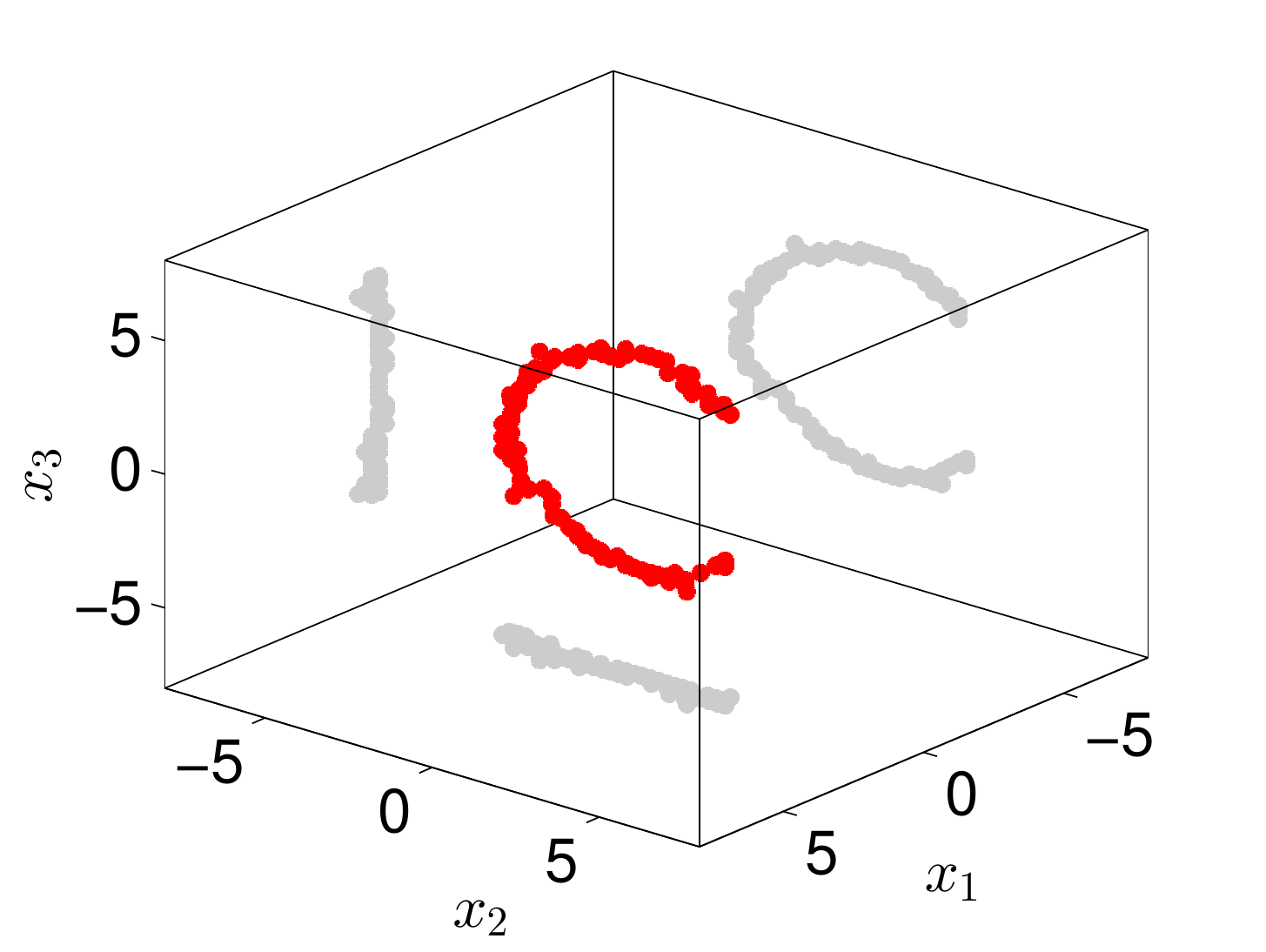}} \subfigure[]{\includegraphics[width=0.32\textwidth]{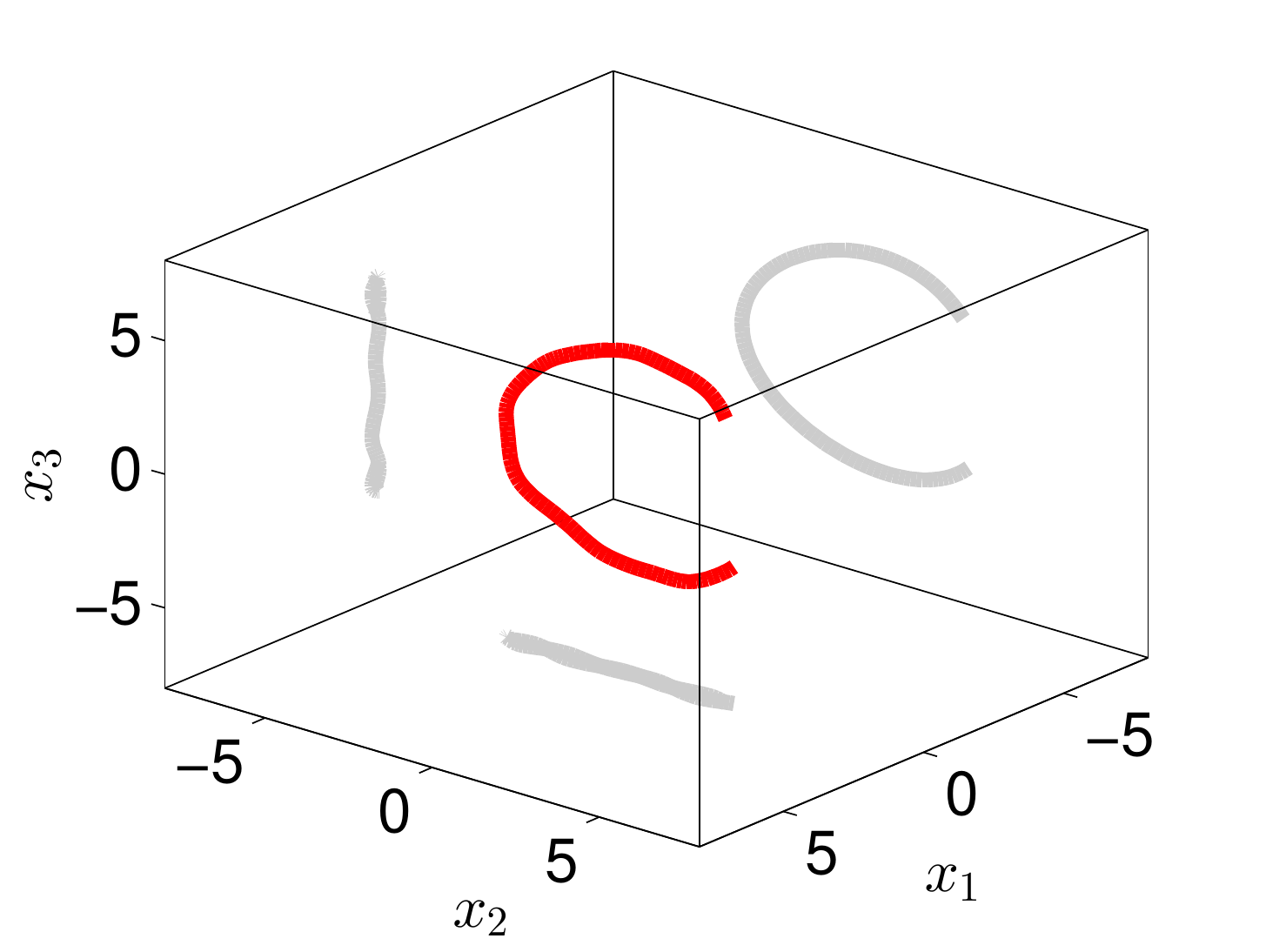}}\\
\subfigure[]{\includegraphics[width=0.32\textwidth]{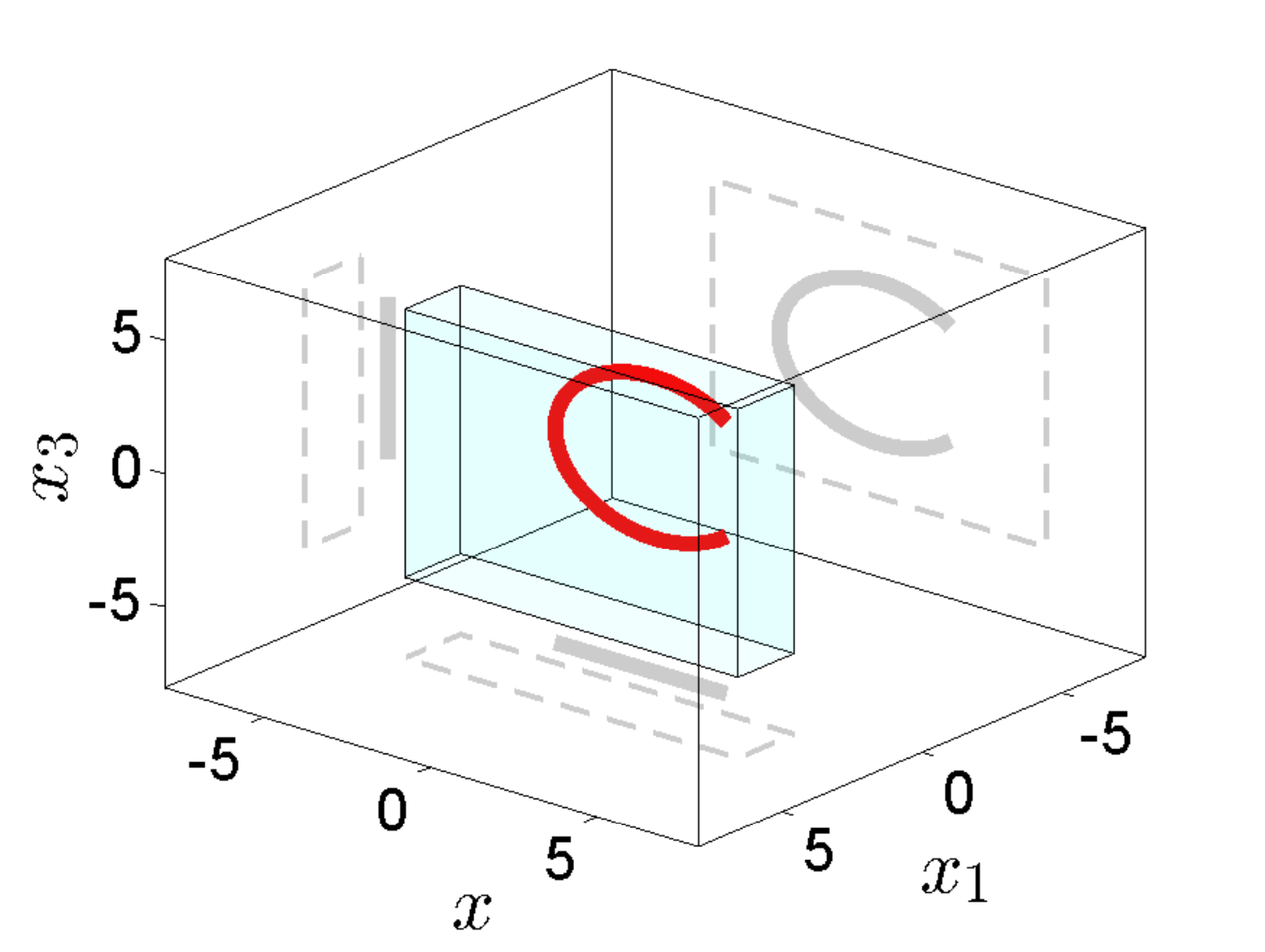}}
\subfigure[]{\includegraphics[width=0.32\textwidth]{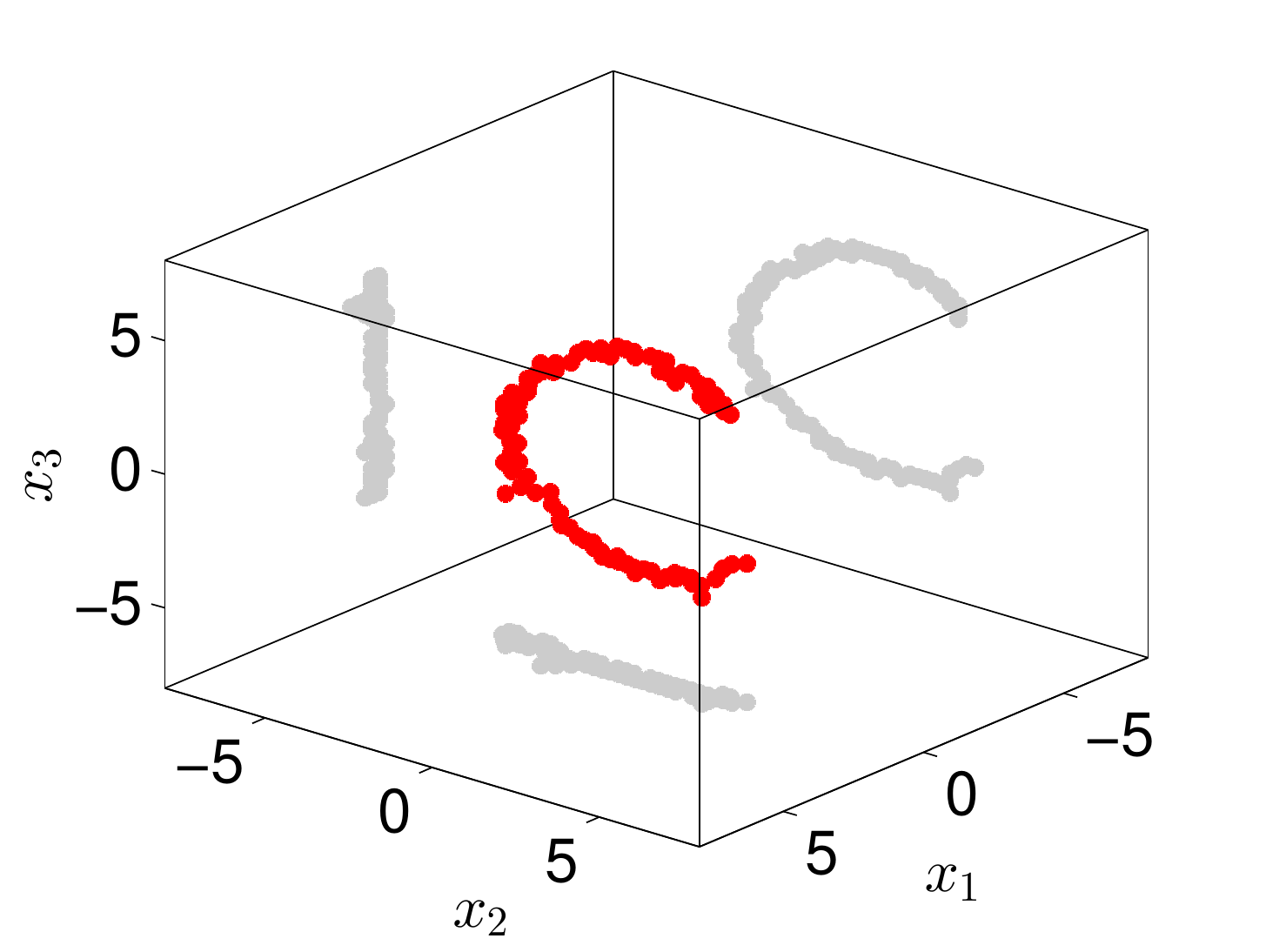}} \subfigure[]{\includegraphics[width=0.32\textwidth]{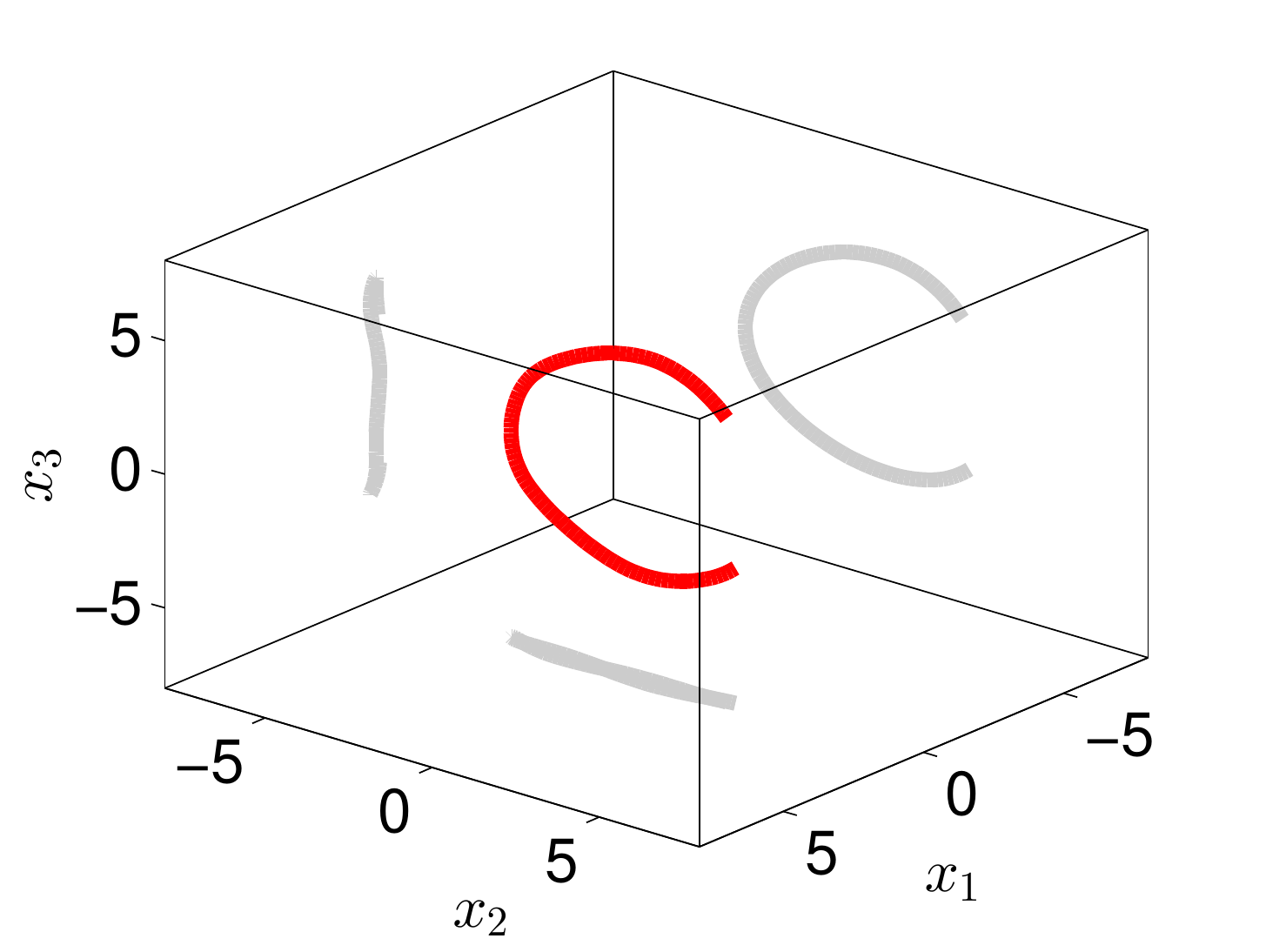}}\\
\caption{Numerical results for Example \ref{exp1}. (a), (d) and (g): true moving trajectories in Case (i), (ii) and (iii), respectively;  (b), (e) and (h): reconstruction of trajectories  in Case (i), (ii) and (iii), respectively; (c), (f) and (i): post-processed trajectories  in Case (i), (ii) and (iii), respectively. }\label{fig.3.1}
\end{figure}

\begin{example}\label{exp1}
Reconstruct the trajectory of a Latin-script letter ``C". Following  Example \ref{exp0}, synthetic wave field were collected at the terminal time $T=10$s with $100$ recording time steps.

This example investigates reconstruction of the trajectories in the respective homogeneous and inhomogeneous media. Figures \ref{fig.3.1}(b), (e) and (d) show scatter plots of trajectory points for all discrete time steps in Case (i), (ii) and (iii). It is observed that there exist some small perturbations in the reconstructed trajectory, because  $5\%$  noise was added to the measurement data. Figures \ref{fig.3.1}(c), (f) and (i) show that the truncated Fourier expansion technique in the post-processing step could yield smooth curves which are close to the exact trajectories.

These numerical results also verified that the reconstructed moving trajectory in the homogeneous medium is very close to that in the inhomogeneous medium, which is due to the fact that the wave scattering generated from the human body is nearly negligible with small angular frequency. Hence, we will only consider the  homogeneous case in the following examples.
\end{example}

%%-------------example2---------------

\begin{figure}[htb!]
 \begin{tabular}{c }
  \hfill\subfigure[]{\includegraphics[  width=0.31\textwidth]{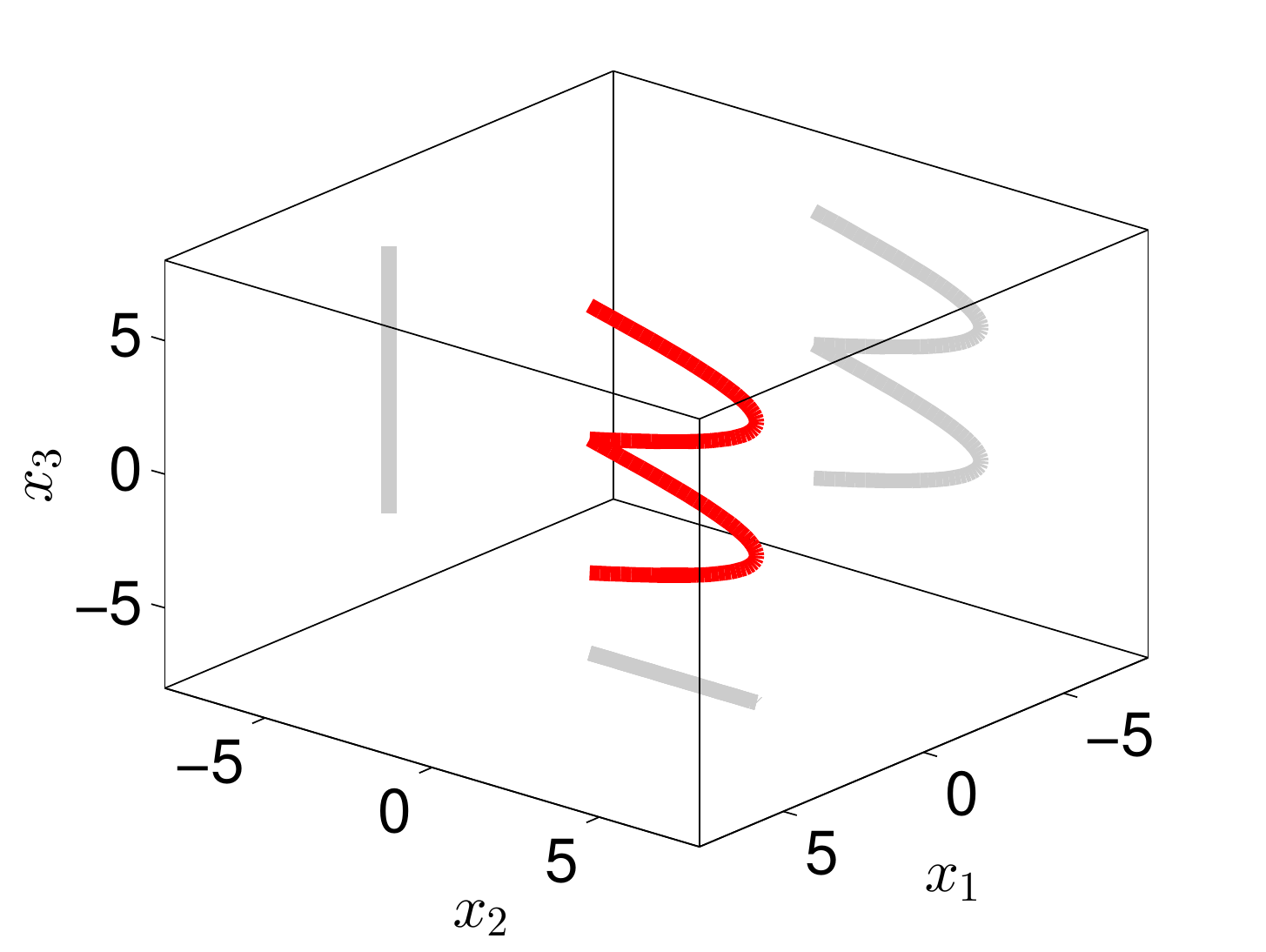}}\hfill
   \end{tabular}
   \begin{tabular}{cc}
   \hfill\subfigure[]{\includegraphics[  width=0.31\textwidth]{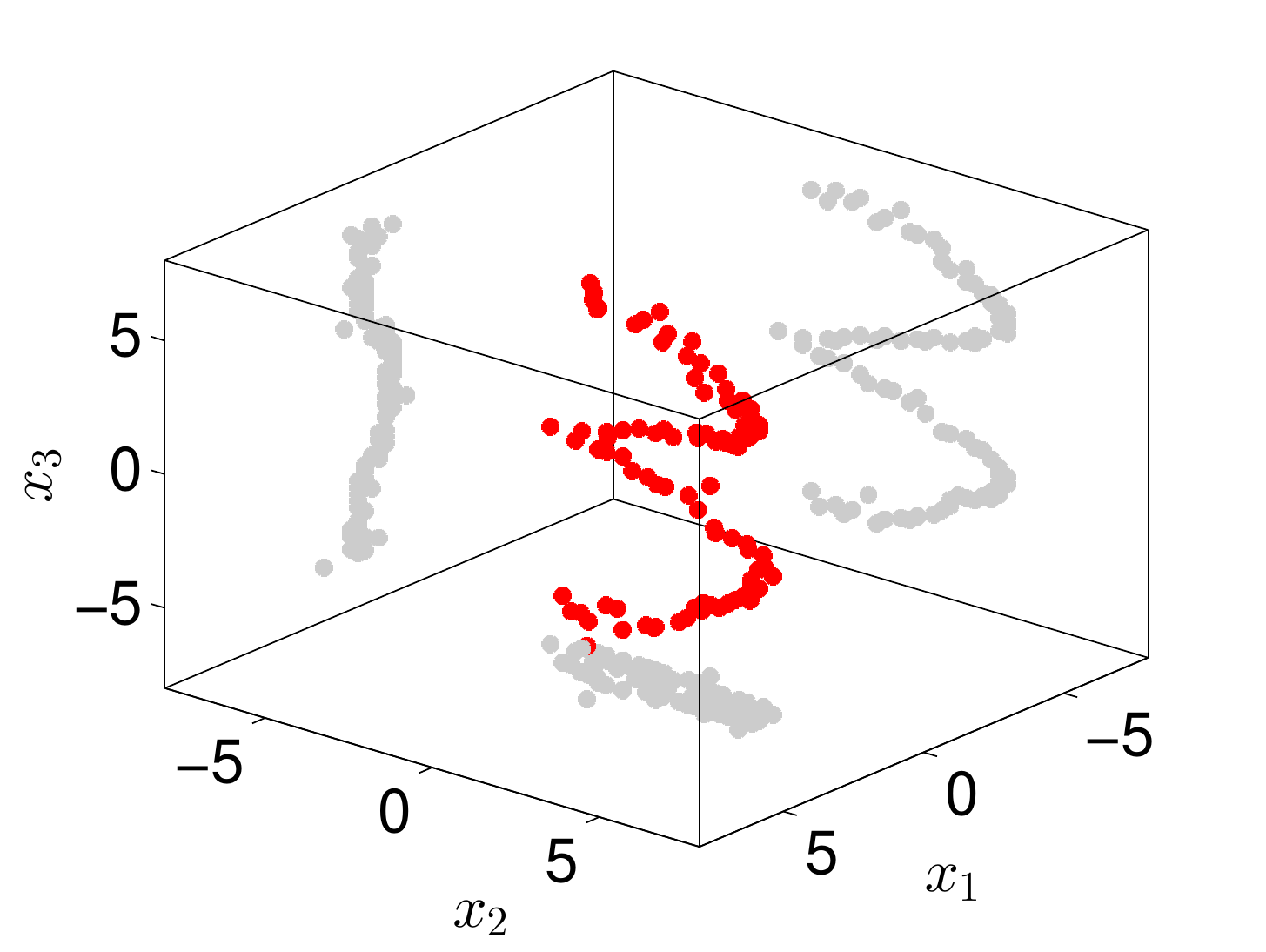}}\hfill
    \hfill\subfigure[]{\includegraphics[  width=0.31\textwidth]{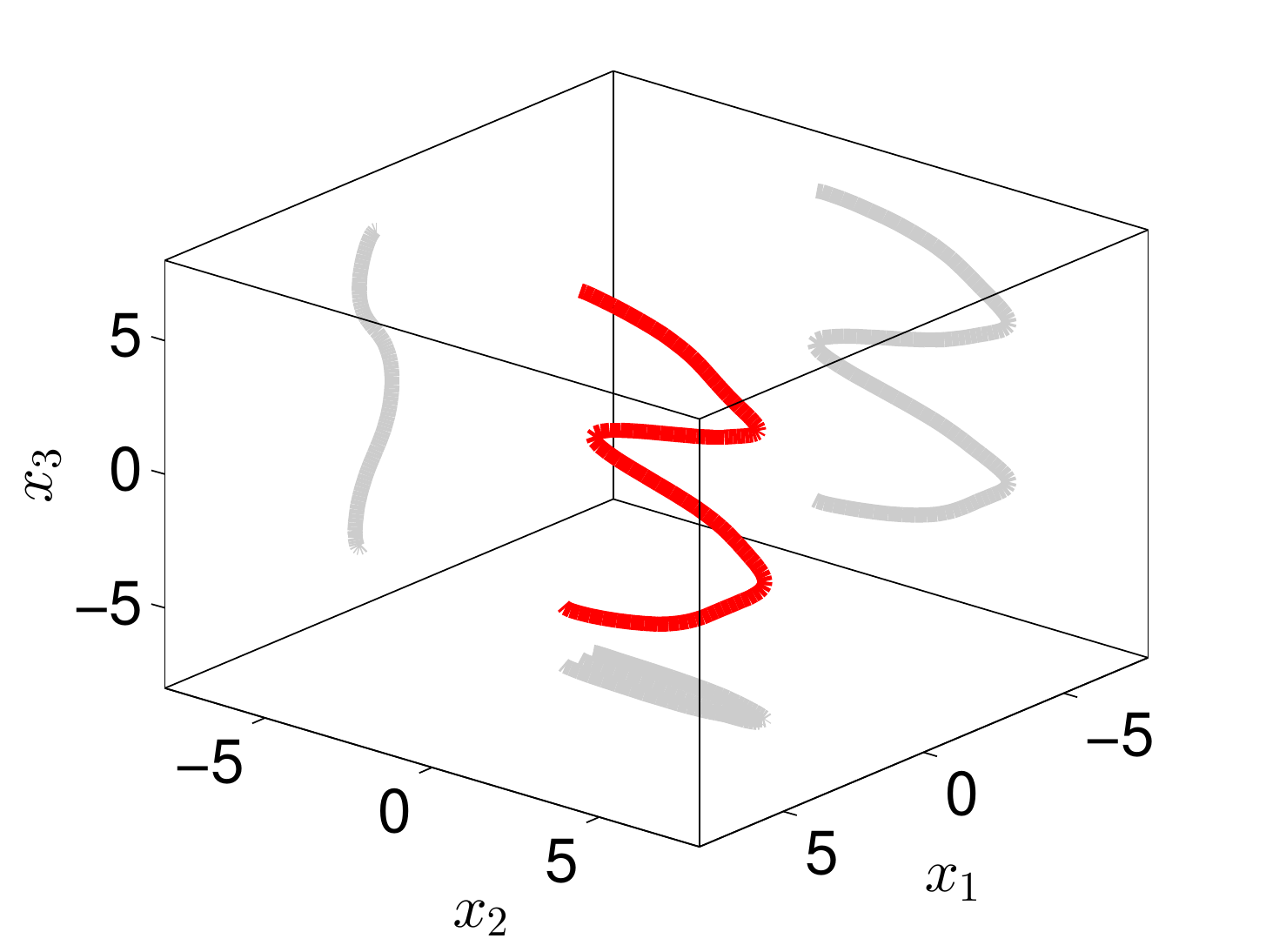}}\hfill\\
     \end{tabular}
   \begin{tabular}{c}
    \hfill\subfigure[]{\includegraphics[  width=0.31\textwidth]{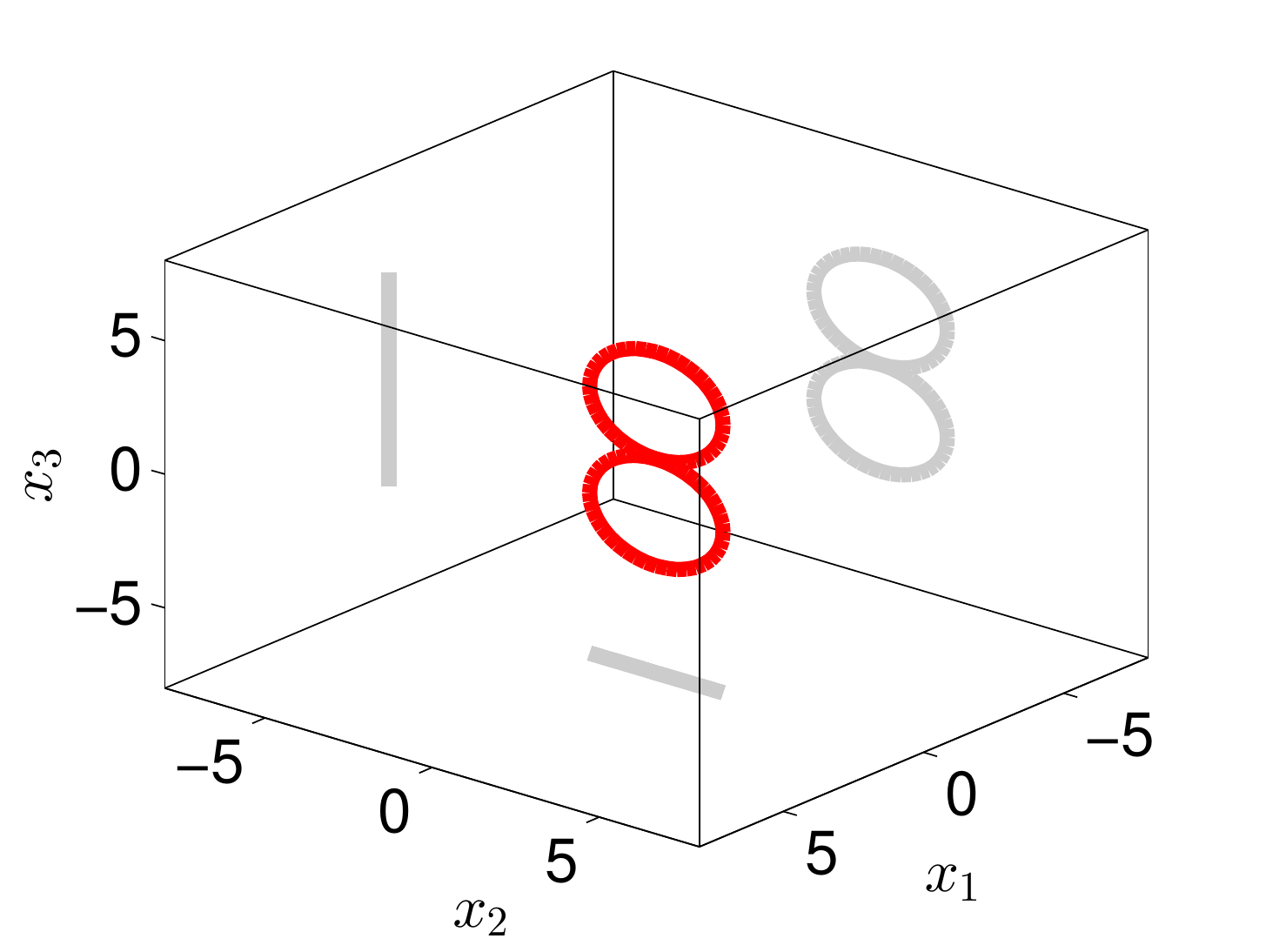}}\hfill
     \end{tabular}
     \begin{tabular}{cc}
      \hfill\subfigure[]{\includegraphics[  width=0.31\textwidth]{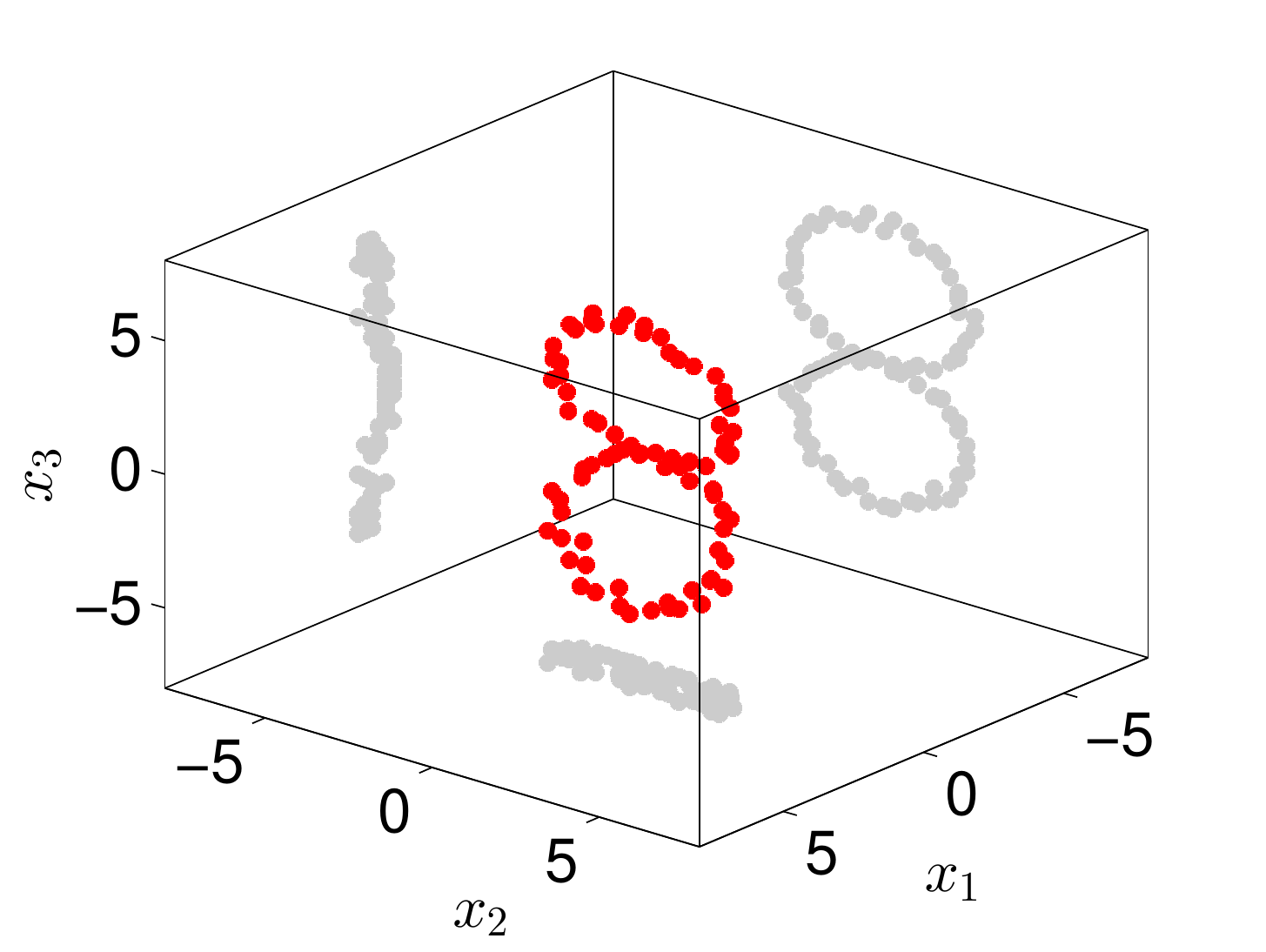}}\hfill
       \hfill\subfigure[]{\includegraphics[  width=0.31\textwidth]{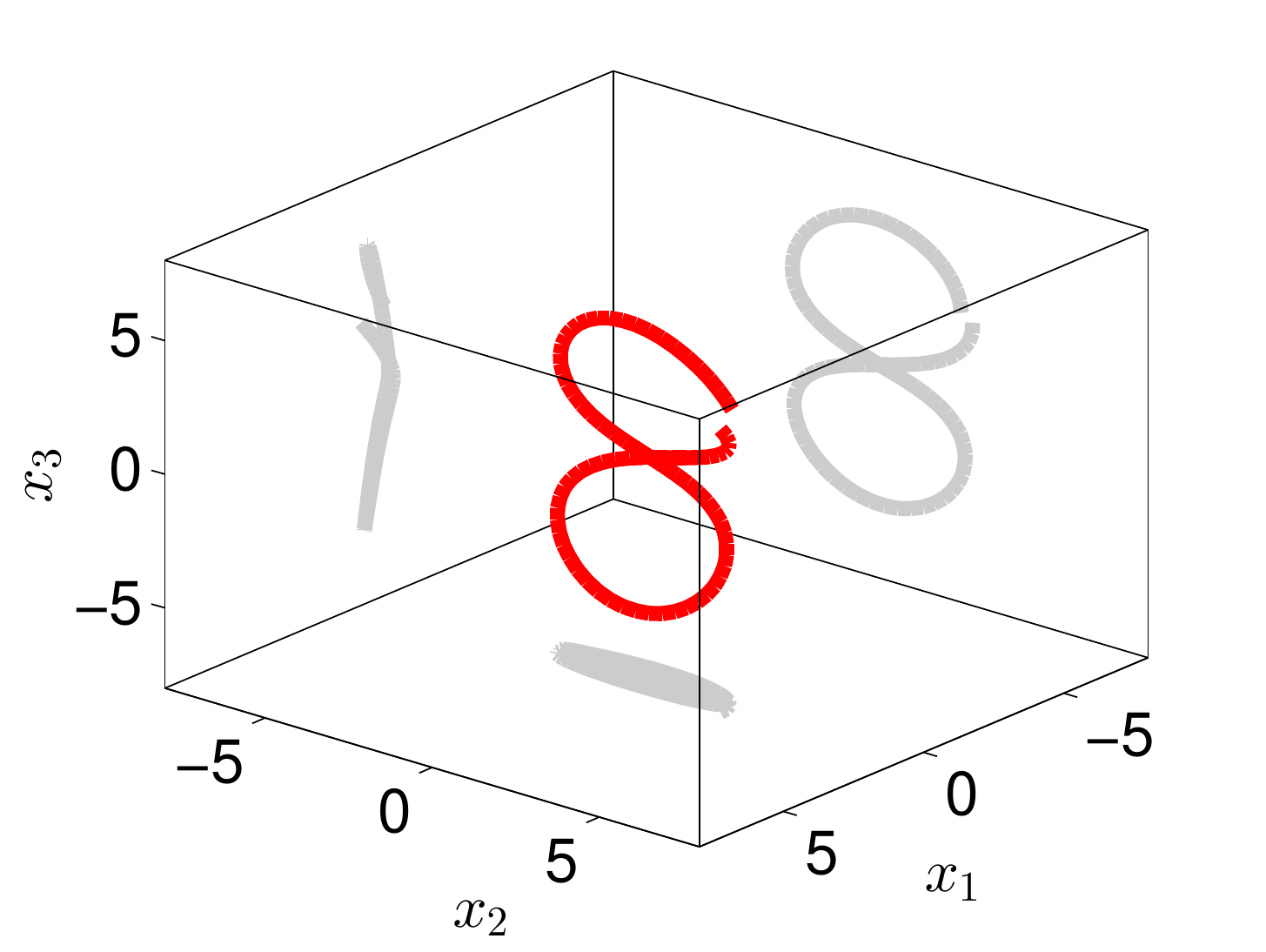}}\hfill\\
        \hfill\subfigure[]{\includegraphics[  width=0.31\textwidth]{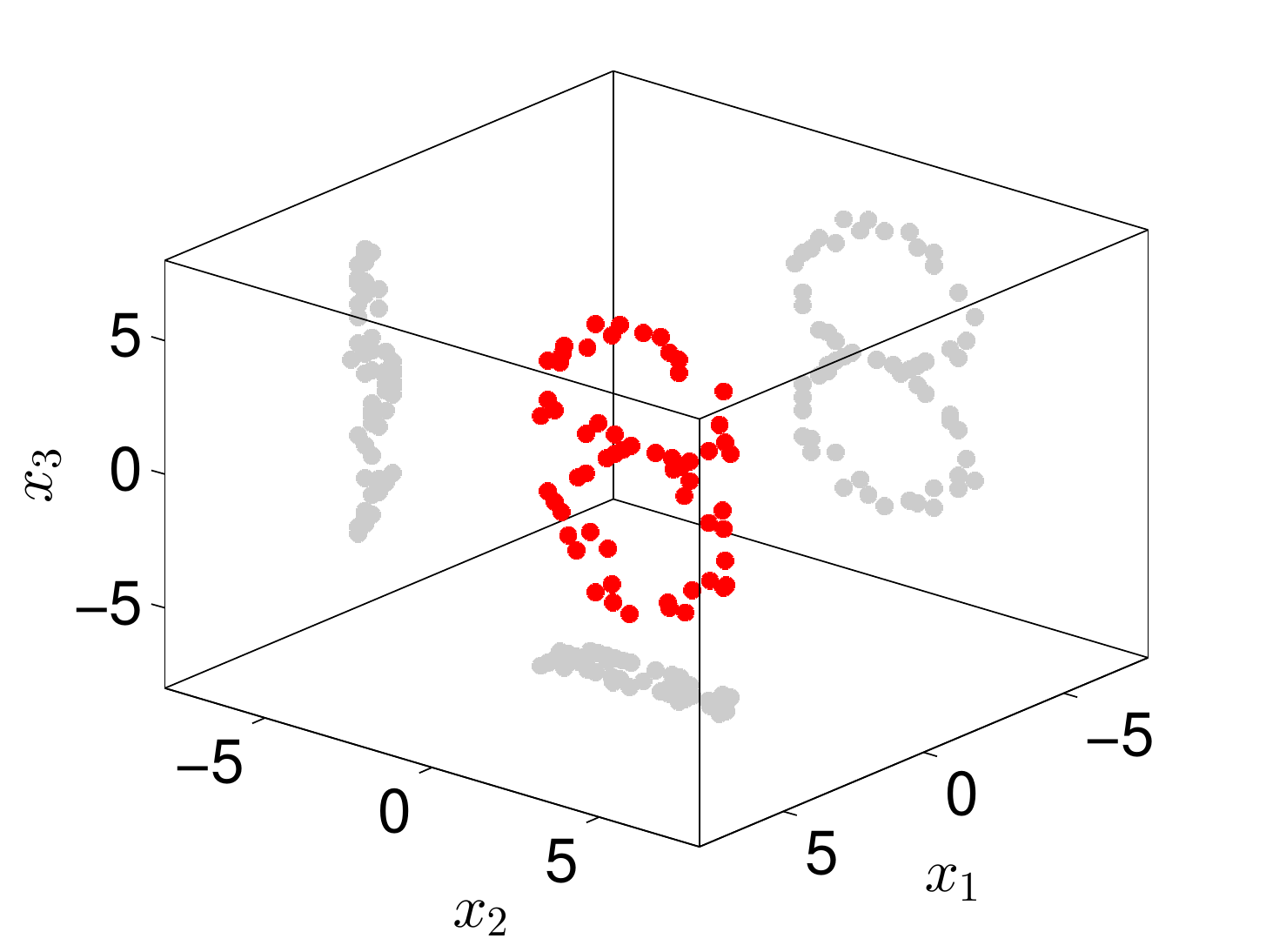}}\hfill
         \hfill\subfigure[]{\includegraphics[  width=0.31\textwidth]{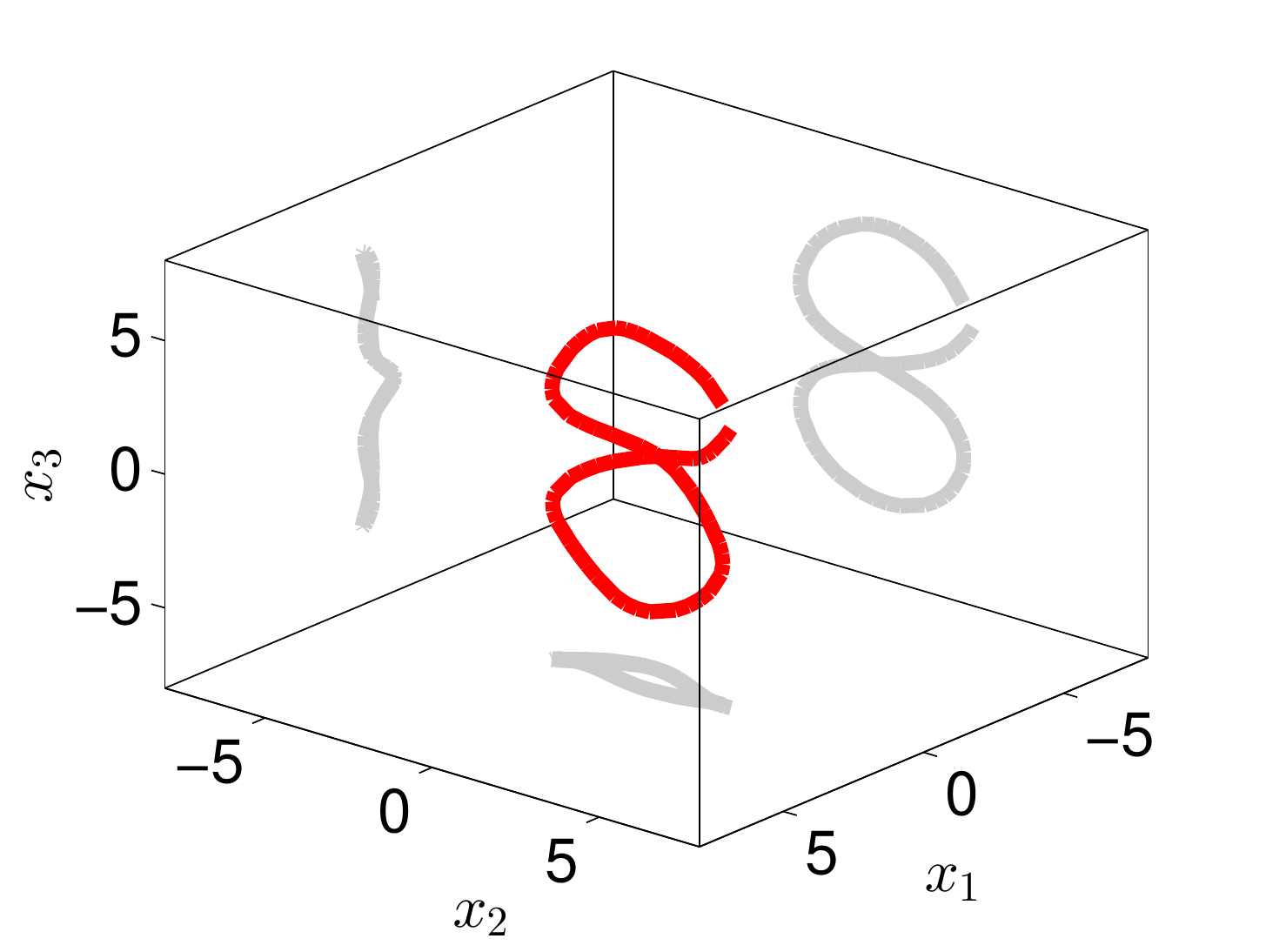}}\hfill\\
          \end{tabular}
        \caption{Reconstruction of the trajectories ``3" and ``8". Left column: true moving trajectories;   Center column: reconstruction of trajectories; Right column: post-processed  trajectories. (e): reconstruction via sequential tuning, (g): reconstruction via parallel tuning.  }\label{fig.3.2}
\end{figure}
\begin{example}\label{exp2}
Reconstruct the trajectories of Arabic numbers ``3" and ``8". Assume that the emitter moves along the following  two paths.

(i)\
\begin{equation*}
  \displaystyle{z_0(t)=\left( 0,\,5\left|\sin\frac{(-5+t)\pi}{5}\right|-2 ,\, 5-t \right)},\quad t\in (0\mathrm{s}, 10\mathrm{s}],
\end{equation*}
which is depicted in Figure \ref{fig.3.2}(a) and served as an approximation of the handwriting Arabic number ``3".

(ii)\
\begin{equation*}
  z_0(t)=\left\{
\begin{aligned}
&  \left(0,\,-2\cos \frac{(t-2)\pi}{2},\, 2 \sin\frac{(t-2)\pi}{2}-2 \right),\quad t\in (0\mathrm{s},3\mathrm{s}]\cup(7\mathrm{s},8\mathrm{s}],\\
&    \left(0,\, 2\cos \frac{\pi t}{2},\,2 \sin\frac{\pi t}{2}+2 \right),\quad t\in (3\mathrm{s},7\mathrm{s}],\\
\end{aligned}
\right.
\end{equation*}
which is depicted in Figure \ref{fig.3.2}(d) and served as an approximation of the handwriting Arabic number ``8".

Note that the trajectory of Arabic number ``3" contains a non-differentiable point such that a sharp corner lies in its neighborhood. In contrast to the smooth parts of the trajectory, the local details such as sharp corners are usually difficult to be reconstructed. The top row of Figure \ref{fig.3.2}  illustrates that our method has the capability of recovering the details of sharp corners, as long as the recording time steps are very small and the local sampling grids are sufficiently fine.

The example of Arabic number ``8" aims to test the capability of our method in resolving a closed trajectory which contains an  intersection point. It can be seen from the last two rows of Figure \ref{fig.3.2} that the intersection point on the trajectory is well identified. For comparison, the reconstructions via the sequential and parallel tuning techniques are shown in Figure \ref{fig.3.2}(e) and (g), respectively. As shown in Figure 5, both of the speed-up techniques could produce satisfactory reconstructions.
\end{example}

%%%-------------example3---------------

\begin{figure}
  \hfill\subfigure[]{\includegraphics[  width=0.3\textwidth]{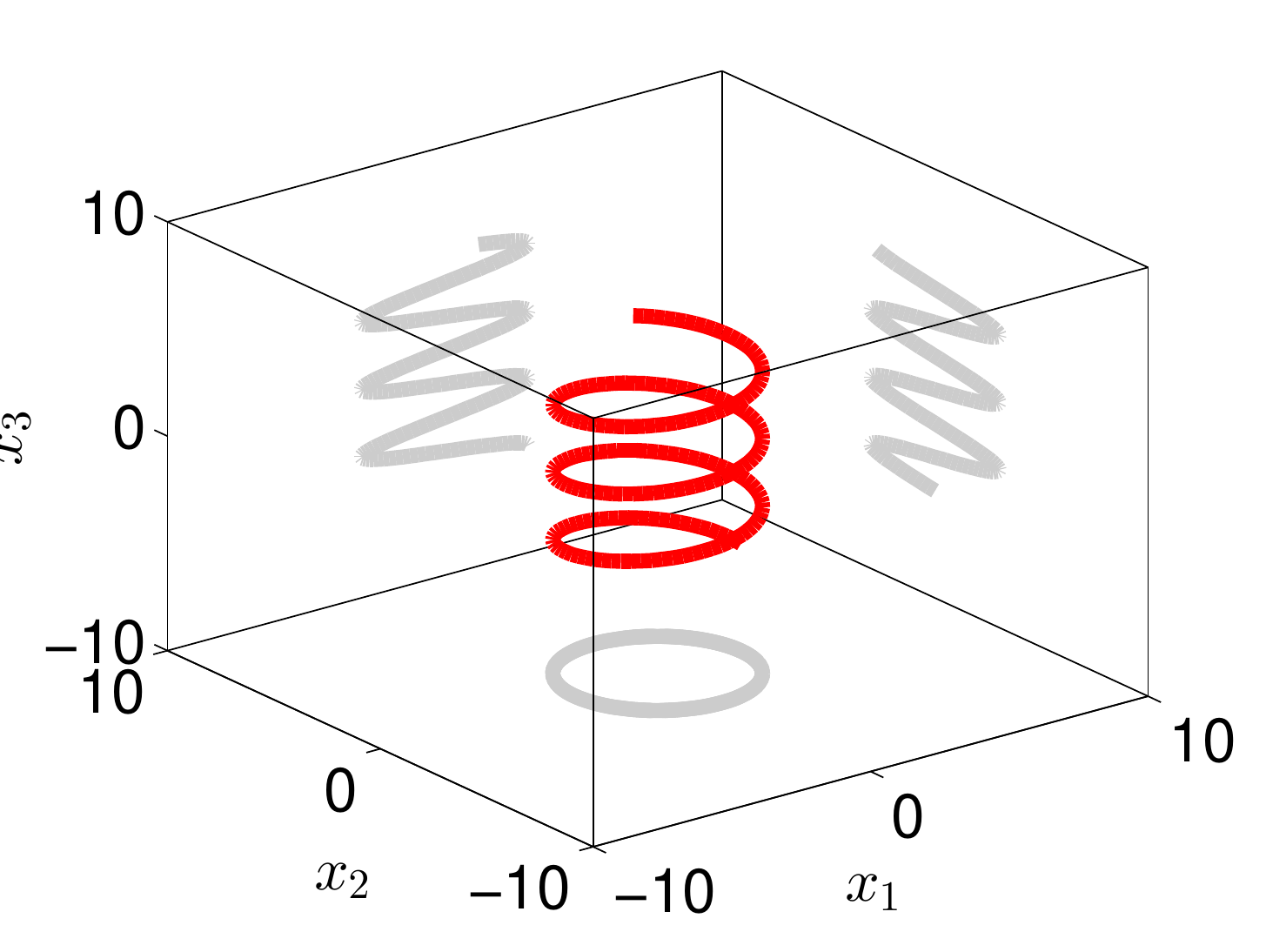}}\hfill
   \hfill\subfigure[]{\includegraphics[  width=0.3\textwidth]{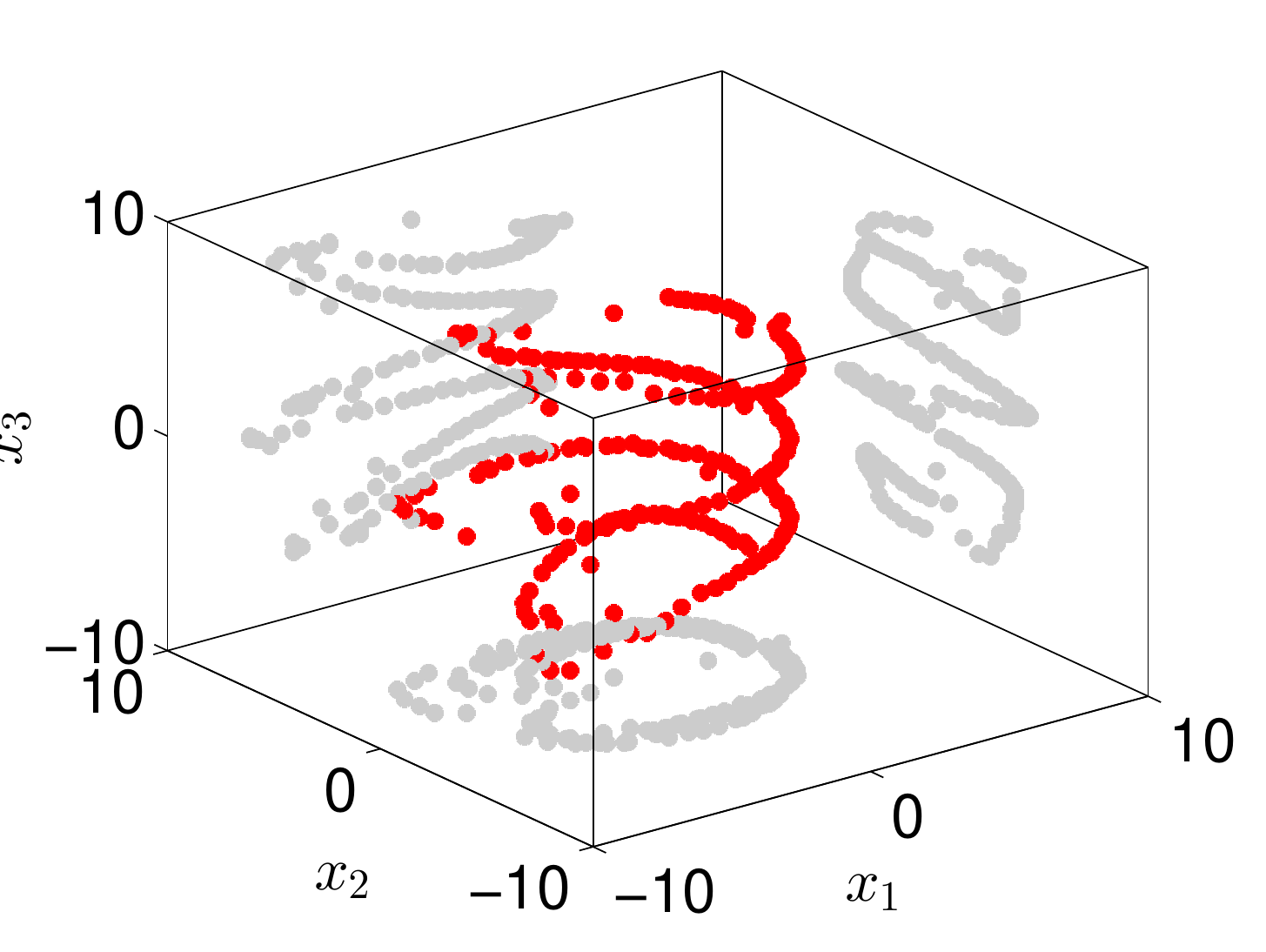}}\hfill
    \hfill\subfigure[]{\includegraphics[  width=0.3\textwidth]{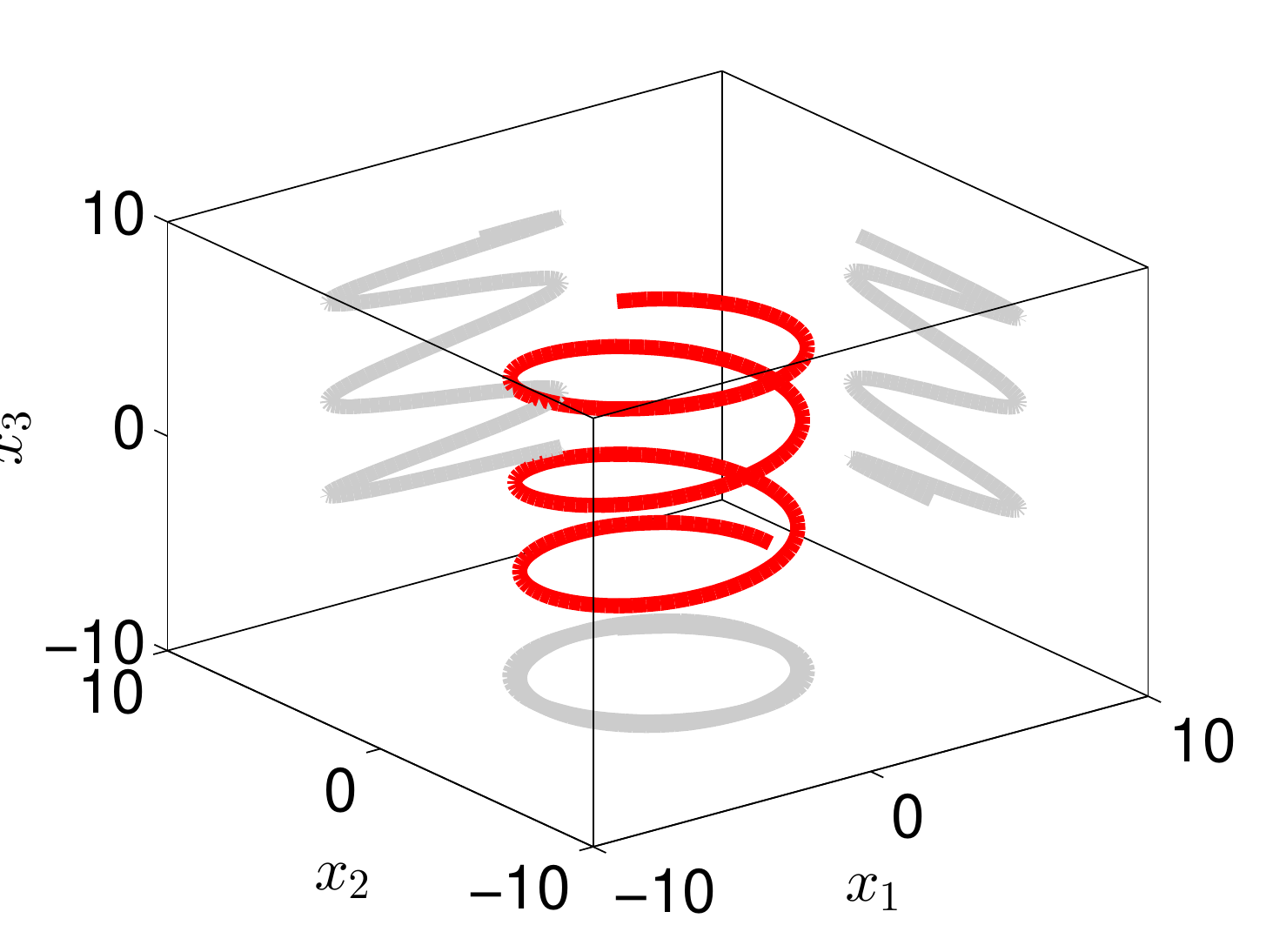}}\hfill\\
      \hfill\subfigure[]{\includegraphics[  width=0.3\textwidth]{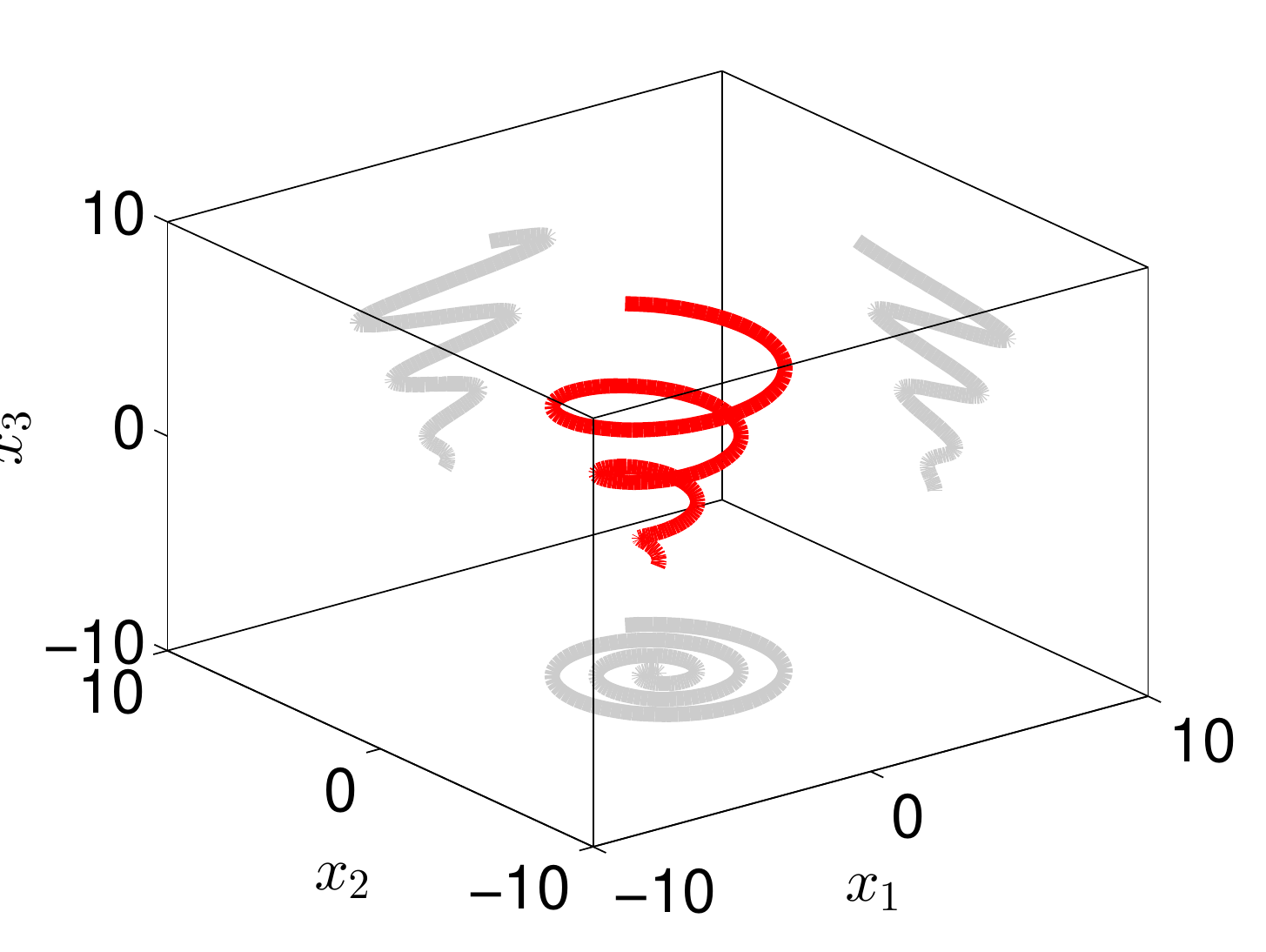}}\hfill
       \hfill\subfigure[]{\includegraphics[  width=0.3\textwidth]{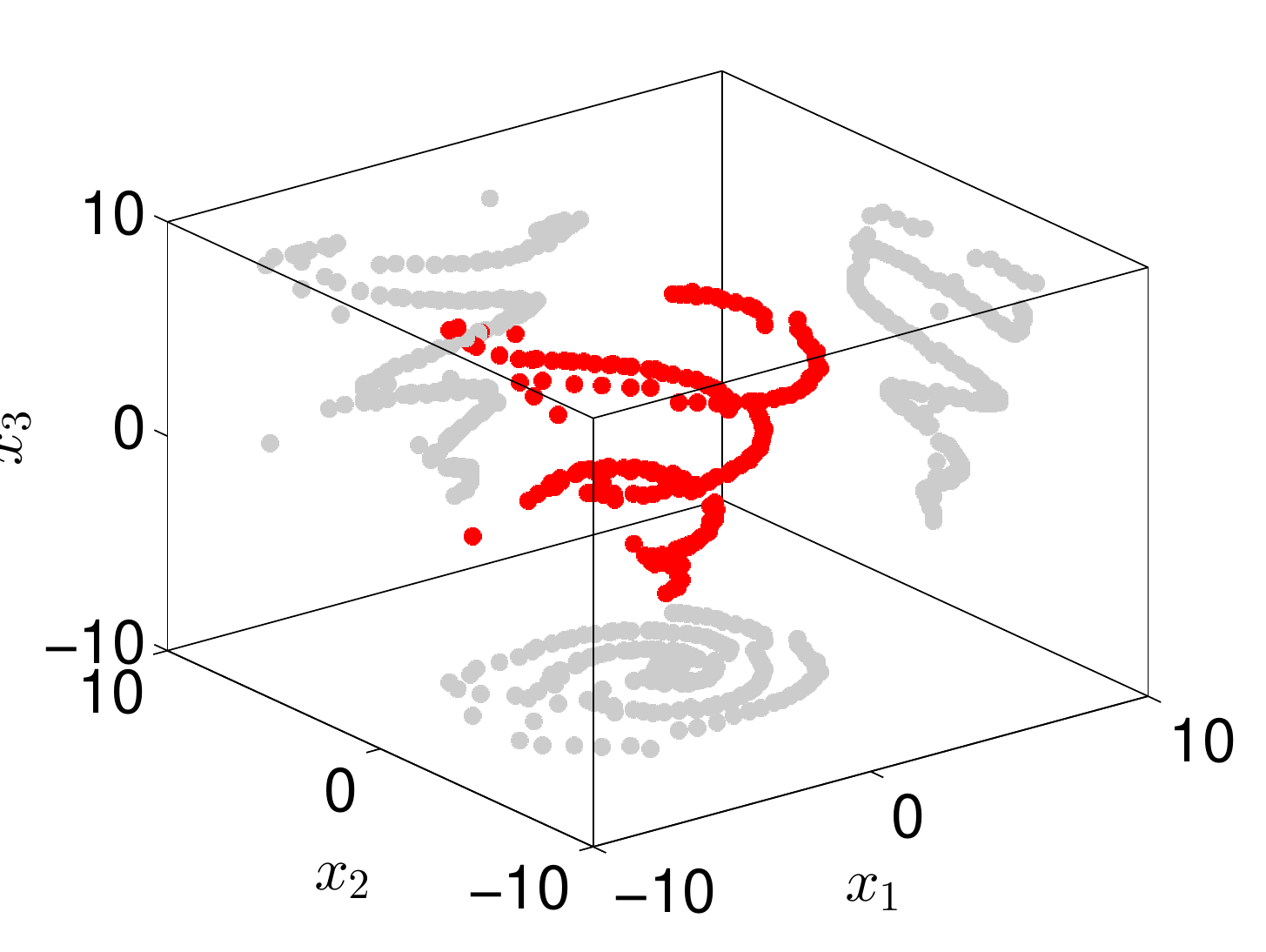}}\hfill
        \hfill\subfigure[]{\includegraphics[  width=0.3\textwidth]{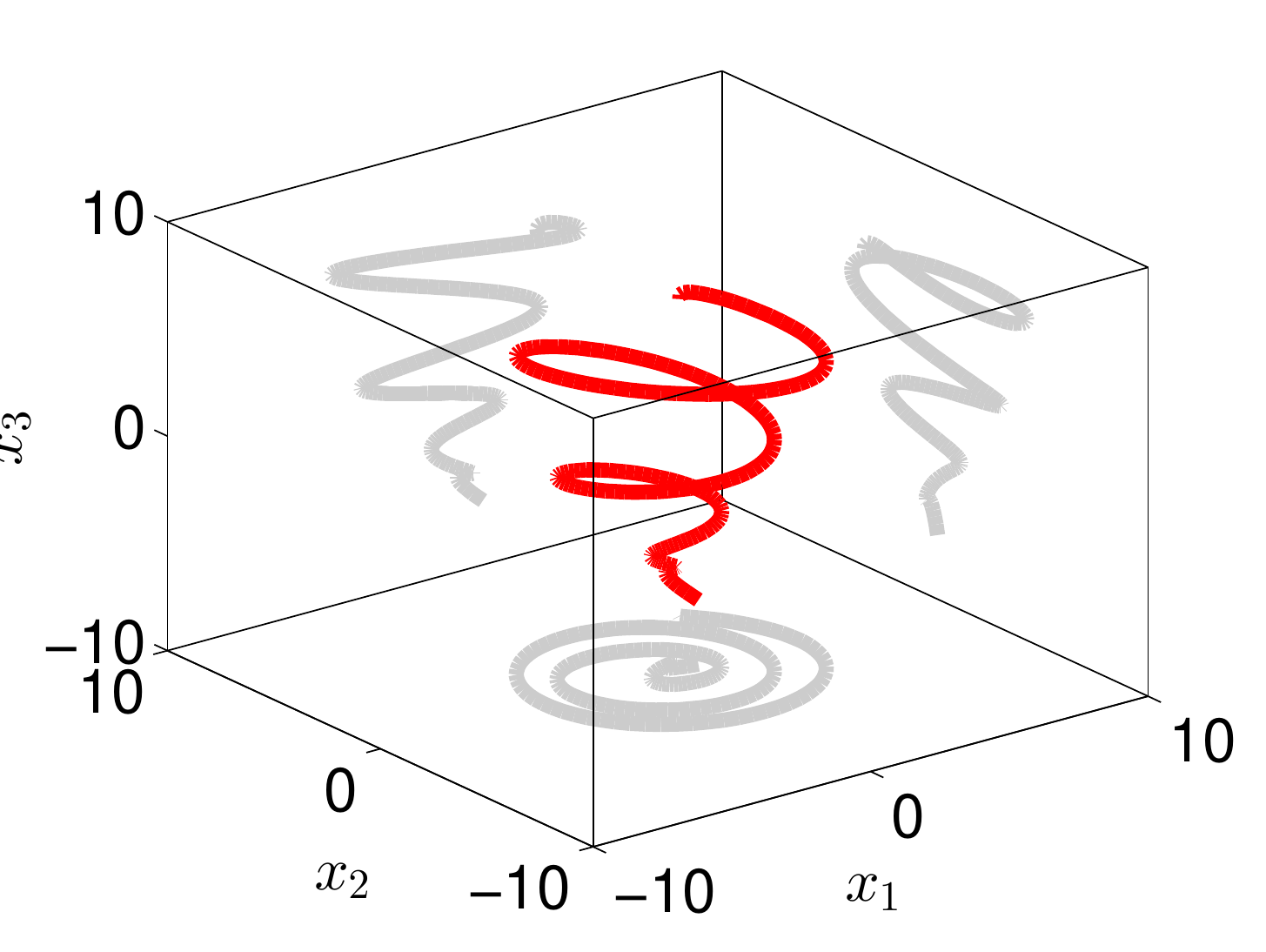}}\hfill\\
   \caption{Reconstruction of spiral-shaped trajectories. Left column: true moving trajectories;  Center column: reconstruction of trajectories; Right column: post-processed  trajectories; Top row: the cylindrical spiral; Bottom row: the conical spiral.}\label{fig.3.3}
\end{figure}

\begin{example}\label{exp3}
Reconstruct a cylindrical/conical spiral. In the previous two examples, we only study the moving trajectories in $x_2$-$x_3$ plane. This example is intended to demonstrate the performance of the algorithm for reconstructing some complicated 3D trajectories. The following two  cases are investigated.

(i)\  Cylindrical spiral: $ z_0(t)=( 3 \cos t, 3\sin t, 0.5\, t-5 ),\, t\in (0\mathrm{s}, \, 20\mathrm{s}].$

(ii)\ Conical spiral: $ z_0(t)=(0.2 \,t \cos t, 0.2 \,t \sin t, 0.5\, t-5),\, t\in (0\mathrm{s},\,  20\mathrm{s}].$

The numerical results are shown in Figure \ref{fig.3.3}, where we post-process the reconstructed cylindrical spiral and conical spiral by truncated Fourier expansion of order 1 and 5, respectively.
As seen in Figure \ref{fig.3.3}, the reconstructions have relatively large perturbations in the half space $x_1<0$. The reason for this behavior is probably that only the limited aperture measurements are available and all the receivers are located in the half space $x_1>0$.

%To illustrate the accuracy, the $x_1, x_2$ and $x_3$ components of the reconstructed trajectories are also plotted, respectively, see Figure \ref{fig.3.3.1}.  Figure \ref{fig.3.3.1}(a) and (d) show that the first components of the reconstructed trajectories $\{z_j\}$ have relatively large errors when $t_j\approx\pi$, $3\pi$ and $5\pi$. Such phenomena occur when $\sin(\omega_0 t)$ has very small value, and at the same time  $\|x-z_0(t)\|$ has very large value. This is  because the relatively small errors  are significantly magnified by recalling \eqref{eq:dd4}.  On the other hand, $\|x-z_0(t)\|$ is relatively  large when $t\approx\pi$ and $5\pi$  with small value of $\sin(\omega_0 t)$. Hence, there exist large errors  when $t\approx\pi$ and $5\pi$ in $x_3$ coordinate.   However, there are no obvious perturbations of reconstructed trajectories in $x_2$ coordinate. The main reason is that $\|x-z_0(t)\|$ is still small , even though $\sin(\omega_0 t)$ is  small when $t\approx\pi$,  $3\pi$ and $5\pi$.
\end{example}

%\begin{figure}
%  \hfill\subfigure[]{\includegraphics[  width=0.3\textwidth]{images/error_cylindarical_x1}}\hfill
%   \hfill\subfigure[]{\includegraphics[  width=0.3\textwidth]{images/error_cylindarical_x2}}\hfill
%    \hfill\subfigure[]{\includegraphics[  width=0.3\textwidth]{images/error_cylindarical_x3}}\hfill\\
%      \hfill\subfigure[]{\includegraphics[  width=0.3\textwidth]{images/error_conical_x1}}\hfill
%       \hfill\subfigure[]{\includegraphics[  width=0.3\textwidth]{images/error_conical_x2}}\hfill
%        \hfill\subfigure[]{\includegraphics[  width=0.3\textwidth]{images/error_conical_x3}}\hfill\\
%   \caption{Componentwise reconstruction of spiral-shaped trajectories. Left column: the first component;  Center column: the second component; Right column: the third component; Top row: the cylindrical spiral; Bottom row: the conical spiral. The green solid lines, the blue dashed lines and the red solid lines denote, respectively,  the true moving trajectories, the reconstruction of trajectories and the post-processed trajectories.}\label{fig.3.3.1}
%\end{figure}

%%-------------example4---------------

\begin{figure}[htb!]
\centering
  \hfill\subfigure[]{\includegraphics[  width=0.45\textwidth]{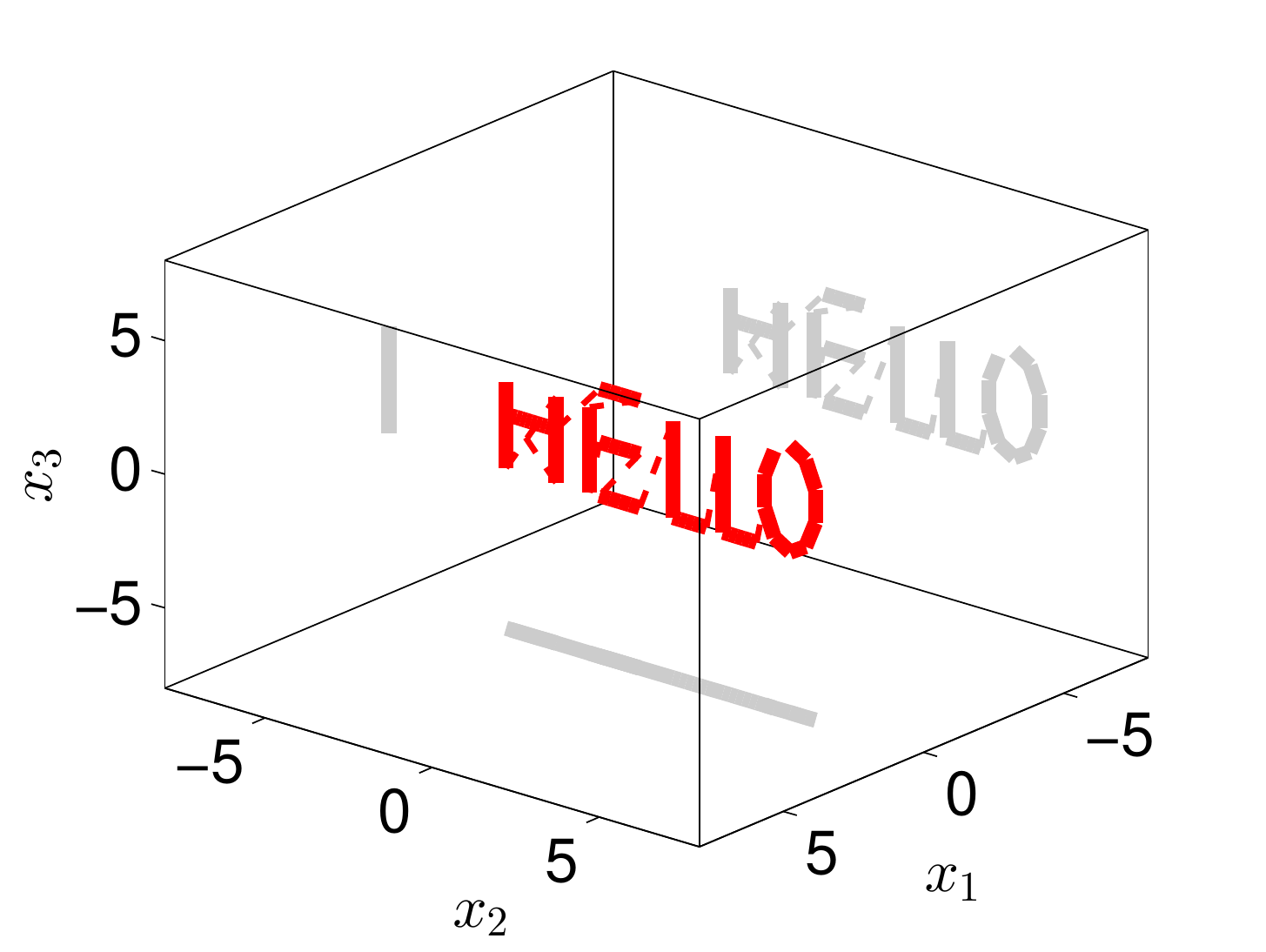}}\hfill
   \hfill\subfigure[]{\includegraphics[  width=0.45\textwidth]{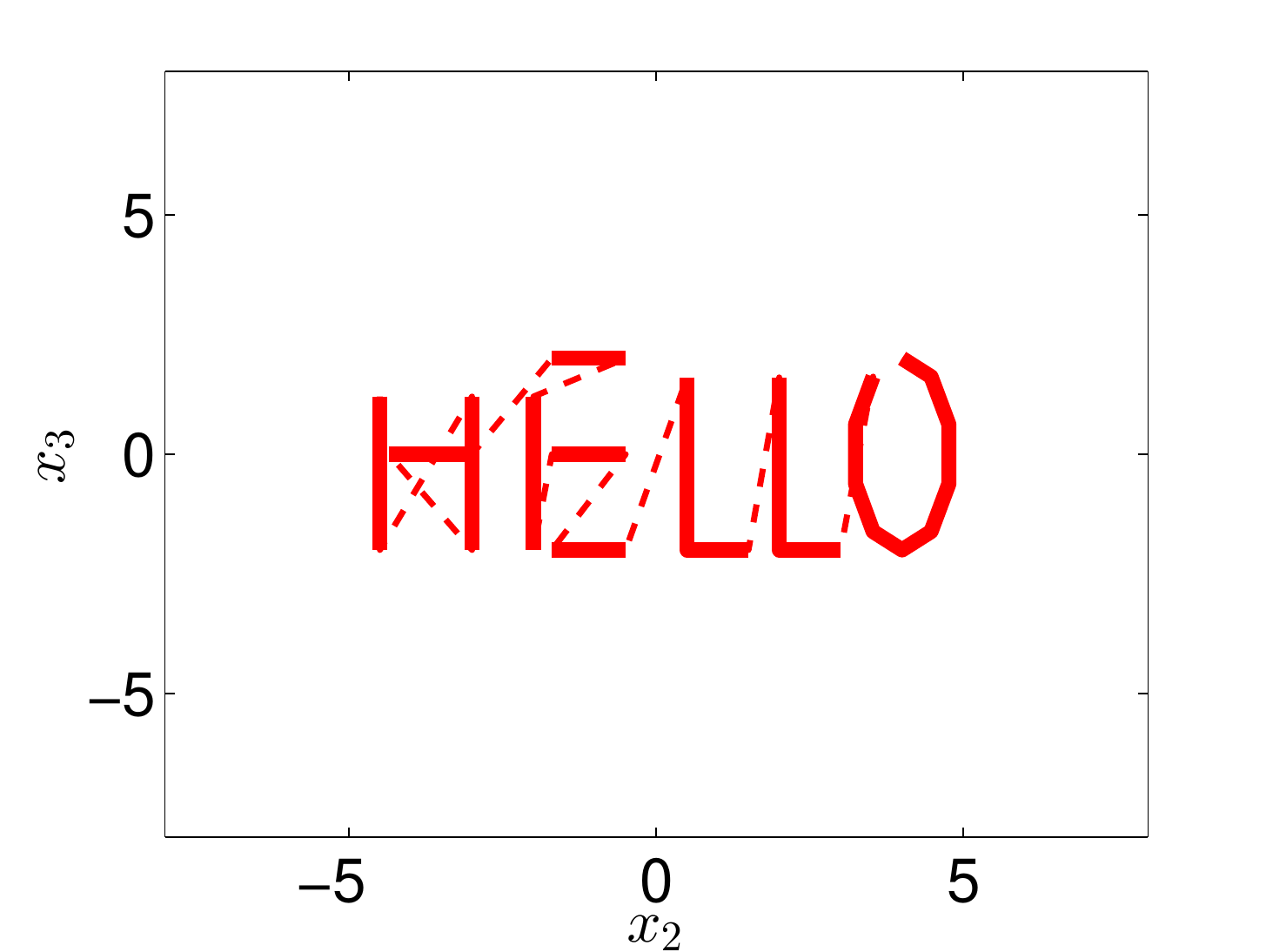}}\hfill\\
     \hfill\subfigure[]{\includegraphics[  width=0.45\textwidth]{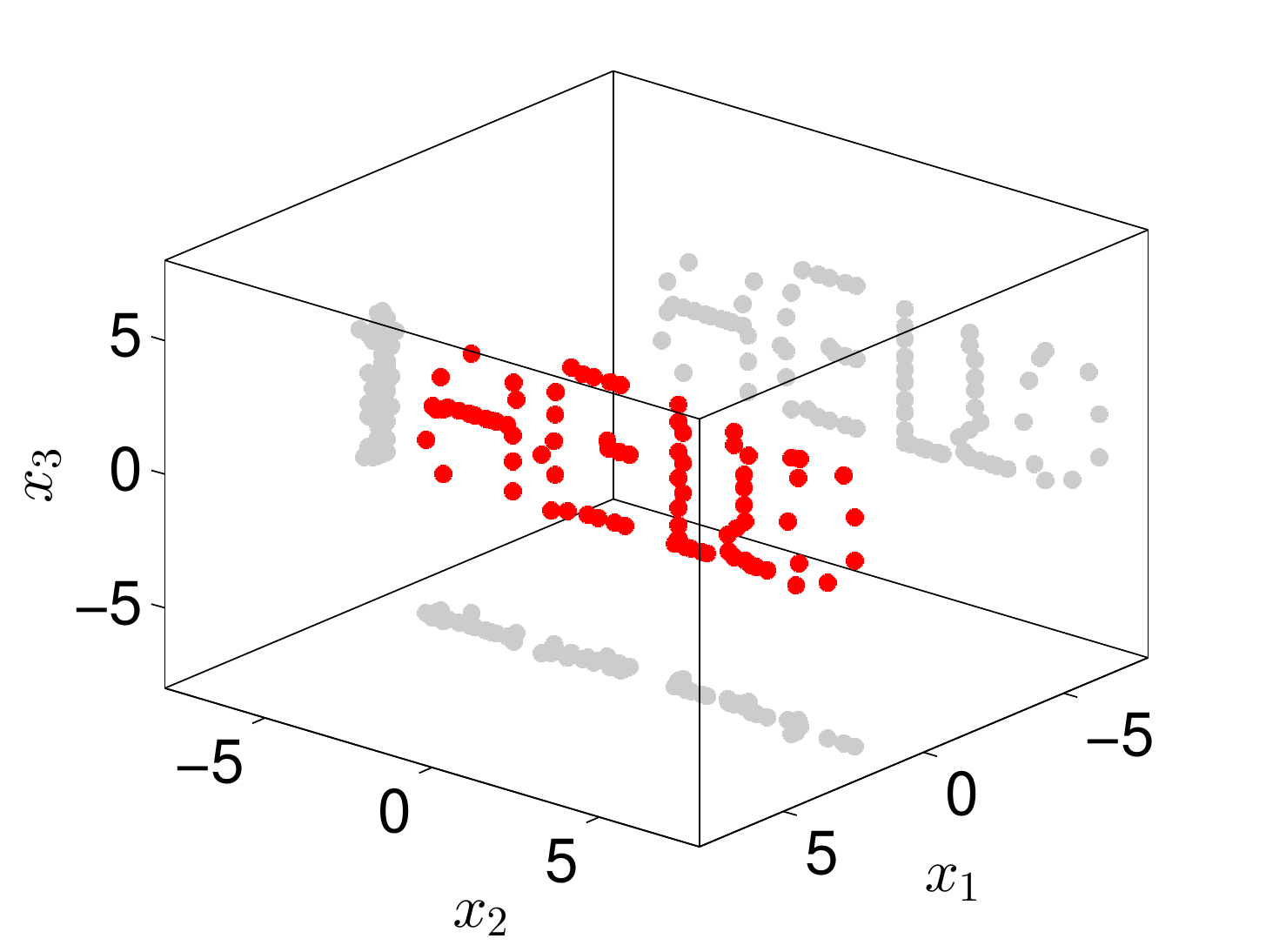}}\hfill
      \hfill\subfigure[]{\includegraphics[  width=0.45\textwidth]{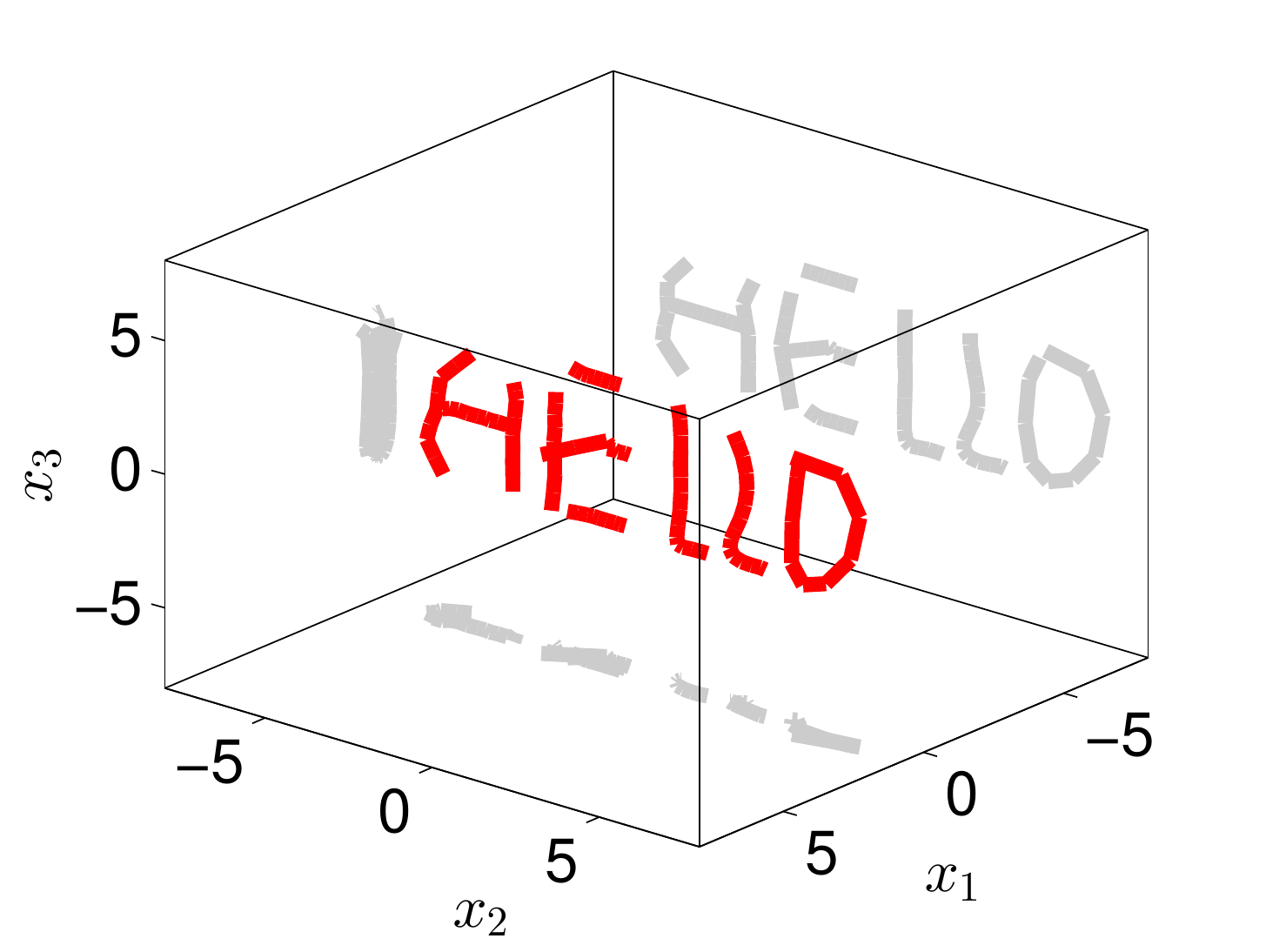}}\hfill
   \caption{Reconstruct the trajectory of a text. (a) (b): true moving trajectory in 3D and 2D view, respectively (the speed of the emitter are $8$m/s and $80$m/s along the solid and dashed lines, respectively), (c): reconstruction of trajectory (d): post-processed trajectory. } \label{fig.3.4}
\end{figure}

\begin{example}\label{exp4}
The last example is devoted to the identification of a motion trajectory consisting of several adjacent Latin-script letters. Assume that someone is wearing an emitter on one of his/her finger and moving the finger to input a text to the computer. Here the motion trajectory is depicted in Figure \ref{fig.3.4}(a) and (b) and served as an approximation of the handwriting ``HELLO''. We want to emphasize that the moving trajectory is continuous although the text ``HELLO" is discontinuous, see Figure \ref{fig.3.4}(b).  The terminal time $T=8$s was employed and $30\%$ noise was added to the synthetic forward data.

It is challenging to distinguish these adjacent letters with large noise. Figure \ref{fig.3.4}(c) shows that our method is insensitive to the noise. Moreover, it seems that the reconstructed points are automatically clustered into several groups and some of the connecting points between these groups are ``missing''. This clustering effect is due to the fact that the velocity of the emitter increased drastically when it travelled between these letters and hence little information was detected. Therefore, it would be more reasonable to
post-process these piecewise clustered points separately rather than output a single smooth moving trajectory. To this end, we first locate the gaps by checking the indices $j$ such that $\|z_j-z_{j+1}\|$ is much larger than the average distance between the successive points for all $\{z_j\}$. According to the location of these gaps, we divide $\{z_j\}$ into several sets. Then the Fourier expansion based post-processing is applied to each set of points separately. Finally, the collection of all the post-processed trajectories is plotted as the result, see Figure \ref{fig.3.4}(d) for the final reconstruction.

%From Figure \ref{fig.3.4}(c), we define $z_j$ by  jumping points  if $|z_j-z_{j-1}|>\|v\|_{\infty}\Delta t$, where $v=8$m/s is the regular writing speed.
%Depending on jumping points, we could divide the reconstructed points into ten parts. Then we piecewisely post-process the reconstructed points by the Fourier expansion technique and obtain a satisfactory result in Figure \ref{fig.3.4}(d). This example illustrates that the proposed algorithm is efficient and robust with respect to the measurement noise.
\end{example}

%\begin{rem}
%As discussed above, sometimes, jumping points may appear in the reconstructions. In this case, two speed-up techniques may not find the global maximum points in the subdomains. So we should replace the subdomains with whole sampling domain $D$ when jumping points appeared.
%\end{rem}

\section*{Acknowledgments}

The work of Y. Guo was supported by the NSF grants of China (No.\,41474102 and No.\,11601107). The work of J. Li was supported by the NSF grant of China  (No.\,11571161), the Shenzhen Sci-Tech Fund (No.\,JCYJ20160530184212170) and the SUSTech Startup fund.  The work of H. Liu was supported by the FRG grants from Hong Kong Baptist University, Hong Kong RGC General Research Funds (12302415 and 405513) and the NSF grant of China (No.\,11371115).

\end{document}